\documentclass[10pt,a4paper,oneside,notitlepage]{amsart}

\usepackage{amssymb}
\usepackage{mathrsfs}
\usepackage[english]{babel}
\usepackage[all]{xy}
\usepackage{url}
\usepackage[pdftex,unicode]{hyperref} 

\theoremstyle{plain}
\newtheorem{Theo}{Theorem}[section]
\newtheorem{Lem}[Theo]{Lemma}
\newtheorem{Cor}[Theo]{Corollary}
\newtheorem{Pro}[Theo]{Proposition}
\newtheorem*{InTheo}{Theorem}
\newtheorem*{InLem}{Lemma}

\theoremstyle{definition}
\newtheorem{Def}[Theo]{Definition}
\newtheorem{Ex}[Theo]{Example}

\theoremstyle{remark}
\newtheorem{Rem}[Theo]{Remark}
\newtheorem{Not}[Theo]{Notation}
\newtheorem{Exp}[Theo]{Explanation of the tables}

\newcommand{\dual}[1]{#1^{\vee}}
\DeclareMathOperator\Hom{Hom}
\DeclareMathOperator\End{End}
\DeclareMathOperator\Sym{Sym}
\DeclareMathOperator\Id{Id}

\title{The Galois closure for rings and some related constructions}
\author{Alberto Gioia \\
Fakult\"at f\"ur Mathematik, University of Vienna, Austria.\\
{\lowercase{\url{alberto.gioia@univie.ac.at}}}}
\date{February 4, 2015}

\begin{document}

\begin{abstract}
Let $R$ be a ring and let $A$ be a finite projective $R$-algebra of rank $n$.
Manjul Bhargava and Matthew Satriano have recently constructed an $R$-algebra
$G(A/R)$, the Galois closure of $A/R$. Many natural questions were asked at the
end of their paper. Here we address one of these questions, proving the
existence of the natural constructions they call intermediate $S_n$-closures. We
will also study properties of these constructions, generalizing some of their
results, and proving new results both on the intermediate $S_n$-closures and on 
$G(A/R)$.
\end{abstract}

\maketitle


\section{Introduction}\label{galois-closure-introduction}

Let $K\to L$ be a finite separable field extension, and let $f$ in $K[Z]$ be
such that $L\cong K[Z]/(f)$. A Galois closure of $L$ over $K$ is a minimal
Galois extension of $K$ containing $L$. Equivalently, it is a field extension
$M$ of $K$, containing $L$, minimal with the property that $f$ splits into
linear factors in $M[Z]$.

Now assume that $f$ has degree $n$, and the Galois group of $f$ is $S_n$, the
symmetric group on $n$ letters. In this case we can construct a Galois closure
as follows: let $K_0$ be $K$, and let $f_0$ be $f$. Given $K_i$ and $f_i$ we
define $K_{i+1}$ as $K_i[X_{i+1}] /(f_i(X_{i+1}))$. Denote by $x_{i+1}$ the
class of $X_{i+1}$ in $K_{i+1}$. Let $f_{i+1}$ be the quotient of $f_i$ by $Z -
x_{i+1}$ in $K_{i+1}[Z]$. The assumption that the Galois group of $f$ is $S_n$
guarantees that for $i= 0,\ldots, n$ the ring $K_i$ is a field, and that $K_n$
is the field we wanted to construct. In particular $x_1,\ldots, x_n$ are the
roots of $f$ in $K_n$.

Let $R$ be a ring (commutative, with identity) and let $A$ be a locally free
$R$-algebra (commutative, with identity) of rank $n$ (see Definition
\ref{def-rank-n}). Manjul Bhargava and Matthew Satriano in \cite{BhargSat}
defined an $R$-algebra $G(A/R)$, which generalizes the Galois closure of an
$S_n$-extension of $K$. In \cite[Question $4$]{BhargSat} they asked whether it
is possible to construct algebras $G^{(i)}(A/R)$ for $i = 1, \ldots, n$, with
$G^{(n)}(A/R) = G(A/R)$, generalizing the intermediate $K_i$ in the construction
above.

In this paper we construct such algebras, which we call $m$-closures, for all
$0\leq m \leq n$. It is sometimes more convenient to use a different
description, which we will call $S$-closure, with $S$ an arbitrary finite set.
In the case of the intermediate $K_i$ for fields, this means that we label the
roots using the set $S$ instead of $\{1,\ldots, i\}$. This will be made precise
in Section \ref{sec-S-closures}.

We will start by recalling some preliminary results on locally free modules and
algebras in Section \ref{galois-closure-preliminaries}. In particular, we 
present the definitions of generic element, and of characteristic polynomials 
of endomorphisms, in a finite locally free module of constant rank. In Section 
\ref{sec-S-closures} we will give the definition, an explicit construction, and 
prove some basic properties of the $S$-closures.

We will then prove the ``product formula'' (Theorem \ref{theo-prod-formula}).
This formula is a generalization of theorem $6$ in \cite{BhargSat}, and
expresses the $S$-closure of a product of $R$-algebras of finite rank in terms
of $T$-closures of the factors, for various subsets $T$ of $S$. This will be
proved in Section \ref{sec-product-formula}. In the same section we will also
prove some consequences of this formula.

In Section \ref{sec-furt-gen}, we will relate the $S$-closures to certain
constructions defined in \cite{FonctNorme} by Daniel Ferrand.

After that, in Section \ref{sec-monogenic} and Section
\ref{galois-closure-examples}, we will study special cases, giving examples and
explicit computations. In a paper currently in preparation, we will study the 
action of the group of automorphisms of the set $S$ on the $S$-closure of an 
algebra, comparing it with the action of the Galois group of a Galois extenion 
of fields.


\section{Preliminaries}\label{galois-closure-preliminaries}

In this section we will collect results needed to define $S$-closures. Nothing
is new, the results are included for convenience. All rings and algebras 
considered, in this section and later on, are associative, commutative, and 
have an identity element, except when otherwise specified.

\begin{Def}\label{def-rank-n}
Let $R$ be a ring. Let $n\geq 0$. Let $M$ be an $R$-module. Following
\cite[Footnote $2$]{BhargSat} we call $M$ an \emph{$R$-module of rank $n$} if it
is finitely generated, projective, and locally free of constant rank $n$.
Similarly an $R$-algebra of rank $n$ is an $R$-algebra that is of rank $n$ as 
an $R$-module.
\end{Def}

It makes sense to give a short description of the definition of characteristic 
polynomials of endomorphisms of modules of rank $n$, since they will play a 
fundamental role in the constructions we will consider. We follow \cite[Chapter 
$4$]{GaloisSchemes} for the definition of the trac: when $M$ is a finitely 
generated projective $R$-module, and $N$ is any $R$-module, we have an 
isomorphism
\begin{equation}\label{isom-trace}
\begin{array}{cccc}
\Phi\colon & N\otimes_R\dual{M} & \to & \Hom_R (N,M)\\
& n\otimes f & \mapsto & \Big( x\mapsto f(x) n \Big)
\end{array}
\end{equation}
where $\dual{M}$ is the $R$-module $\Hom_R (M,R)$. (See \cite[Section
$4.8$]{GaloisSchemes}.) Taking $N = M$ one can define the \emph{trace map} 
$s_1\colon \End_R(M)\to R$, as the composition of $\Phi^{-1}$ with the
natural map $M\otimes_R \dual{M} \to R$ sending $m\otimes f$ to $f(m)$.

Since the exterior power $\bigwedge^i M$ of an $R$-module of rank $n$ is an
$R$-module of rank $\binom{n}{i}$, given $f\in \End (M)$, one can define the
\emph{$i$-th trace} of $f$, denoted $s_i(f)$, to be the trace of the
endomorphism induced by $f$ on $\bigwedge^i M$. The $n$-th trace of $f$ is
called the \emph{determinant} of $f$. For every endomorphism $f$ of $M$ define
the \emph{characteristic polynomial} of $f$, denoted $P_f(X)$, as the
determinant of the endomorphism $(\Id \otimes X - f\otimes \Id)$ of $M\otimes_R
R[X]$. Note that this is a monic polynomial in $R[X]$, and that we have
\[
P_f(X) =  \sum_{i= 0}^{n} (-1)^i s_{i}(f) X^{n-i}.
\]

If $A$ is a finite projective $R$-algebra, then we have a map $A\to \End(A)$
sending $a\in A$ to multiplication by $a$. Composing this map with the $s_i$
we get maps $A\to R$, which we denote again $s_i$; in this case $s_n$ will be
called \emph{norm map}. Similarly, we obtain the characteristic polynomial of 
$a$, denoted $P_a(X)$.

\begin{Rem}\label{rem-Cay-Ham}
Note that Cayley-Hamilton theorem holds, i.e. for all endomorphisms $f$ of an
$R$-module $M$ of rank $n$, we have $P_f(f) = 0$. In fact, this holds for free
modules and if an element of $\End (M)$ is zero locally at every prime of $R$
then it is zero. Moreover if $A$ is an $R$-algebra, then $P_a(a)$ is zero.
\end{Rem}

The following construction is only a formal variant of the usual tensor power of
an $R$-algebra. This form will be useful later to simplify the notation, 
especially in Section \ref{sec-product-formula}.

\begin{Def}\label{def-tensor-power-S}
Let $A$ be an $R$-algebra. For any finite set $S$ the \emph{tensor power of $A$
indexed over $S$} is an $R$-algebra $A^{\otimes S}$ given with a map
$\varepsilon_s\colon A\to A^{\otimes S}$ for every $s\in S$, such that for any
$R$-algebra $B$ with a map $\zeta_s\colon A \to B$ for every $s\in S$ we have a
unique map $\varphi\colon A^{\otimes S} \to B$ making the following diagram
commutative for every $s\in S$:
\[
\xymatrix{
A \ar[r]^{\varepsilon_s} \ar[d]_{\zeta_s}&
A^{\otimes S} \ar[dl]^{\varphi} \\
B
}
\]
\end{Def}

\begin{Rem}\label{rem-tensor-power}
Note that any bijection $\{1,\ldots, \#S\} \to S$ induces a unique isomorphism 
$A^{\otimes \#S} \to A^{\otimes S}$ compatible with the natural maps.
\end{Rem}

Finally, we introduce generic elements, which we will use to construct the 
$S$-closure explicitly. To do that we define the \emph{symmetric algebra}. For 
more detail see Bourbaki \cite[Chapitre $\mathrm{III}$, \S $6$]{BourbakiAlg13}.

\begin{Def}\label{def-symmetric-algebra}
Let $M$ be an $R$-module. The \emph{symmetric algebra} of $M$ is an $R$-algebra
$\Sym M$ given with an $R$-module map $\varepsilon \colon M\to \Sym M$ such 
that the map
\[
\Hom_{\mathrm{R}\textrm{-alg}}(\Sym M, A) \to \Hom_R(M, A)
\]
\[
\varphi\mapsto \varphi\circ \varepsilon
\]
is a natural bijection.
\end{Def}

\begin{Def}\label{def-gen-elem}
Let $M$ be a finitely generated projective $R$-module. Tensoring the identity of
$M$ with the natural map $\dual{M} \to \Sym(\dual{M})$ we get a morphism
$M\otimes_R \dual{M} \to M\otimes_R \Sym(\dual{M})$. Composing with the
isomorphism $\Phi^{-1}$ defined in \eqref{isom-trace} (with $N$ equal $M$), we
get an $R$-module map
\[
\End(M) \to M\otimes_R \Sym(\dual{M}).
\]
We call the \emph{generic element} of $M$, denoted $\gamma_M$ or simply
$\gamma$, the image of $\Id_M$ via this map.
\end{Def}

\begin{Ex}\label{ex-gen-elem}
Let $M$ be a free $R$-module with basis $e_1,\ldots, e_n$. Let $X_1,\ldots, X_n$ 
be the dual basis. Then the polynomial ring $R[X_1,\ldots, X_n]$ with the map 
$X_i \mapsto X_i$ has the universal property of $\Sym (\dual{M})$. The generic 
element of $M$ is then
\[
\gamma = \sum_{i=1}^n e_i \otimes X_i
\]
in $M\otimes_R\Sym(\dual{M})$. In general, the symmetric algebra of an 
$R$-module $M$ is a graded algebra, with the degree $0$ part isomorphic to $R$ 
and the degree $1$ part isomorphic to $M$. The generic element is always an 
element of $ M\otimes_R \dual{M}$. We will denote by $\Sym^n M$ the degree $n$ 
part of $\Sym M$.
\end{Ex}

Before proving the main lemma about the generic element, we will need the
following result.

\begin{Lem}\label{lem-bidual}
Let $M$ be a locally free $R$-module. Then the map
\[
M\to \dual{(\dual{M})}
\]
\[
m \mapsto (f \mapsto f(m))
\]
is injective. If moreover $M$ is finitely generated, then the map is an
isomorphism.
\end{Lem}
\begin{proof}
By localizing at primes of $R$, we can reduce to free modules. Here the
statement becomes: if $x$ is such that $f(x)$ is zero for all $f\in \dual{M}$
then $x$ is zero, and this is clear for a free module $M$.

If $M$ is finitely generated, then locally the map is also an isomorphism, 
hence it is an isomorphism globally.
\end{proof}

The following important property is what makes the generic element useful.

\begin{Pro}\label{pro-gen-elem}
Let $M$ be a finitely generated projective $R$-module. Let $R'$ be any
$R$-algebra. Then the map
\[
\Hom_{R\textrm{-\emph{Alg}}}(\Sym(\dual{M}), R') \to M' = M\otimes_R R'
\]
\[
\varphi\mapsto (\Id_M\otimes \varphi)(\gamma)
\]
is a natural bijection.
\end{Pro}
\begin{proof}
Recall from \eqref{isom-trace} that for all $R$-modules $N$ we have the
isomorphism $N\otimes_R \dual{M} \cong \Hom_R(M,N)$. By \ref{lem-bidual} we have
$M\cong \dual{(\dual{M})}$. So we already know we have natural bijections
\[
M' \cong \Hom_R(\dual{M}, R') \cong \Hom_{R\textrm{-}\mathrm{alg}}(\Sym
(\dual{M}),R').
\]
Hence $\Sym(\dual{M})$ represents the functor sending an $R$-algebra $R'$ to
the set $M'$. Taking $R' = \Sym(\dual{M})$ we get that $\gamma$ is mapped to 
the identity in $\End{(\Sym(\dual{M}))}$ by definition, hence by Yoneda's lemma 
the map has the required form.
\end{proof}

\begin{Ex}\label{ex-gen-elem-two}
Let $M$ be free of rank $n$. Write $\gamma$ as $\sum_i e_i \otimes X_i$, with
the notation introduced in \ref{ex-gen-elem}. For $x$ in $M\otimes_R R'$ the 
image of $x$ in $\Hom_{R\textrm{-\emph{Alg}}}(\Sym(\dual{M}), R')$ is
the unique $R$-algebra morphism $\varphi_x$ sending $e_i$ to $X_i(x)$ for all 
$i$. Hence, if $A$ is a free $R$-algebra, the map $\mathrm{Id}\otimes\varphi_x$ 
is the evaluation map $A[X_1,\ldots, X_n]\to R'$ sending a polynomial $f$ to
$f(s_1,\ldots, s_n)$, where $(s_1,\ldots, s_n)$ is the element of ${R'}^{ n}$
representing $x$ in the chosen basis.
\end{Ex}

\begin{Rem}\label{rem-gen-elem-charpoly}
Let $A$ be an $R$-algebra of rank $n$. Let $R\to R'$ be an $R$-algebra, and
let $a$ be in $A' = A\otimes_R R'$. Denote by $\varphi_a$ the image of $a$ in
$\Hom_{R\textrm{-\emph{Alg}}}(\Sym(\dual{A}), R')$. Let $P_\gamma$ be the
characteristic polynomial of $\gamma$ in the $\Sym(\dual{A})$-algebra
$A\otimes_R\Sym(\dual{A})$. The map $\Id_A\otimes \varphi_a \colon A\otimes_R
\Sym(\dual{A})[X]\to A'[X]$ sends $P_\gamma$ to the characteristic polynomial of
$a$. In fact, this is true for free algebras as follows easily from Example
\ref{ex-gen-elem-two}, so for an algebra of rank $n$ it is true locally at every
prime of $R$, and hence the claim follows.
\end{Rem}


\section{\texorpdfstring{$S$-closures}{S-closures}}\label{sec-S-closures}

In this section we are going to define $S$-closures and prove their existence
giving an explicit construction. We will use characteristic polynomials and
generic elements, defined in Section \ref{galois-closure-preliminaries}. We
start with the definition.

\begin{Def} \label{universal-property}
Let $A$ be an $R$-algebra of rank $n$, and let $S$ be a finite set. An
$R$-algebra $A^{(S)}$ given together with $R$-algebra maps $\alpha_s\colon A
\to A^{(S)}$ for every $s\in S$, is an \emph{$S$-closure of $A$} if for
all $R\to R'$ and all $a\in A\otimes_R R'$ the polynomial
\[
\Delta_a(X) = \prod_{s\in S} \big(X-(\alpha_s\otimes \mathrm{Id})(a)\big)
\]
divides the characteristic polynomial $P_a$ of $a$ in ${A^{(S)}}\otimes_R
R'[X]$, and the pair $\left(A^{(S)},(\alpha_s)_{s\in S} \right)$ is universal
with this property.

Being universal means that for all $R$-algebras $B$ given with maps
$\beta_s\colon A\to B$ for $s\in S$, if for all $R\to R'$ and $a\in A\otimes_R
R'$ the polynomial $\prod_s(X-(\beta_s\otimes \mathrm{Id})(a))$ divides the
characteristic polynomial of $a$ in $B\otimes_R R'[X]$, then there is a unique
morphism $\varphi\colon {A^{(S)}} \to B$ such that for every $s\in S$ the
following diagram commutes:
\[
\xymatrix{
A \ar[r]^{\alpha_s} \ar[d]_{\beta_s} & A^{(S)}
\ar^{\varphi}[dl]\\
B
}
\]
\end{Def}

\begin{Rem}\label{rem-cons-univ-prop}
From the universal property it follows (by standard argument) that if the
$S$-closure of an $R$-algebra of rank $n$ exists, then it is unique up to a
unique isomorphism. The set $\alpha_S(A) = \{\alpha_s(a) \mid a\in A, s\in S\}$ 
has the same universal property as $A^{(S)}$, hence the $S$-closure is generated 
by $\alpha_S(A)$.
\end{Rem}

\begin{Rem}
One can define the $S$-closure of a scheme $X$ that is finite locally free of
rank $n$ over a scheme $Y$ (see also the introduction of \cite{BhargSat}). We
will limit our study to the affine case.
\end{Rem}

In the rest of the seciton we give an explicit construction of the $S$-closure,
showing it exists for any $R$-algebra of rank $n$ and any finite set $S$, and we
prove some easy consequences of the construction. We will need some results for
the construction.

\begin{Lem}\label{lem-evaluations}
Let $M$ be a locally free $R$-module. Then the map
\[
\begin{split}
M& \to \prod_{\lambda\in \dual{M}} R \\
m& \mapsto (\lambda(m))_{\lambda}
\end{split}
\]
is injective.
\end{Lem}
\begin{proof}
This is just Lemma \ref{lem-bidual}.
\end{proof}

Recall that given an $R$-module $M$ we denote by $\Sym(M)$ the symmetric
algebra of $M$ (see Definition \ref{def-symmetric-algebra}).

\begin{Lem}\label{lem-polys}
Let $M$ be an $R$-module of rank $n$, and let $C$ be an $R$-algebra. Let $t\in
C \otimes_R\Sym(\dual{M})$. Then the following are equivalent:
\begin{enumerate}
\item The element $t$ is zero.
\item For all $\lambda$ in $\Hom_R(\Sym(\dual{M}),R)$ we have $(\mathrm{Id}_{C}
\otimes \lambda)(t)$ is zero in $C$.
\item For all $R\to R'$ and all $R$-algebra morphisms $\varphi\colon
\Sym(\dual{M})\to R'$ we have $(\mathrm{Id}_C \otimes \varphi)(t)$ is zero in
$C\otimes_R R'$.
\end{enumerate}
\end{Lem}
\begin{proof}
If $t$ is zero then both \emph{2} and \emph{3} hold.

Suppose \emph{3} holds. Then we can take $R' = \Sym(\dual{M})$ and $\varphi$
the identity map, so $t$ is zero and \emph{1} and \emph{3} are equivalent.

Suppose \emph{2} holds. Since $\Sym(\dual{M})$ is locally free, the natural map
\[
C\otimes_R \Hom_R\left(\Sym(\dual{M}),R\right)\to
\Hom_C\left(C\otimes_R \Sym(\dual{M}), C\right)
\]
is surjective (because it is surjective at every prime of $R$). Let $\mu$ be an
element of $\Hom_C(C\otimes_R \Sym(\dual{M}), C)$ and let $\nu$ be an element of
$C\otimes_R \Hom_R\left(\Sym(\dual{M}),R\right)$ mapping to $\mu$. Write $\nu$
as $\sum_i c_i \otimes \lambda_i$. Then we have
\[
\mu(t) = \sum_{i} c_i(\mathrm{Id}_{C}\otimes \lambda_i(t))
\]
and this is zero, since for all $i$ we have $(\mathrm{Id}_C\otimes \lambda_i)(t)
= 0$. Then Lemma \ref{lem-evaluations} with $R$ equal $C$ and $M$ equal
$C\otimes_R \Sym(\dual{M})$ implies that $t$ is zero. Hence \emph{1} and
\emph{2} are equivalent.
\end{proof}

We can now prove that the $S$-closure of a rank $n$ algebra exists. We first
introduce the notation we will use in the proof.

\begin{Not}\label{not-ideal}
Let $A$ be an $R$-algebra of rank $n$. We will write $A_{\Sym}$ for $A\otimes_R
\Sym(\dual{A})$. Let $S$ be a finite set. Note that we have an isomorphism
\[
A^{\otimes S}_{\Sym} = (A\otimes_R \Sym(\dual{A}))^{\otimes S} \cong A^{\otimes
S}\otimes_R
\Sym(\dual{A})
\]
where the first tensor power is taken over $\Sym(\dual{A})$ and the last over
$R$. We will use this isomorphism without spelling it out. For $s\in S$ we
denote by the same symbol both the natural map $\varepsilon_s \colon A \to
A^{\otimes S}$ and its base change $\varepsilon_s\colon A_{\Sym}\to
A_{\Sym}^{\otimes S}$.

Let $P_\gamma (X)\in A_{\Sym}[X]$ be the characteristic polynomial of the
generic element. Define the polynomial
\[
\Delta_\gamma = \prod_{s\in S} \big( X - \varepsilon_s(\gamma) \big)
\]
in $A_{\Sym}^{\otimes S}[X]$. Note that since $\Delta_\gamma$ is monic, division
with remainder of $P_\gamma$ by $\Delta_\gamma$ can be done in
$A_{\Sym}^{\otimes S}[X]$ and gives a unique quotient and remainder, with the
remainder of degree less than $\#S$. We write
\[
P_\gamma = \Delta_\gamma Q_\gamma + T_\gamma
\]
and 
\[
T_\gamma = \sum_{0\leq i<\#S} t_i X^i
\]
with $t_i$ in $A_{\Sym}^{\otimes S}$.

In $A^{\otimes S}$ we define the following ideal:
\[
J^{(S)} = \left\langle (\mathrm{Id}_{A^{\otimes S}}\otimes\lambda)(t_i) \mid i
\geq 0, \lambda \in \Hom_R(\Sym(\dual{A}),R)\right\rangle.
\]
We will prove that $A^{\otimes S}/J^{(S)}$ is an $S$-closure for $A$.
\end{Not}

\begin{Pro}[Construction of $A^{(S)}$]\label{pro-gal-constr}
Let $A$ be an $R$-algebra of rank $n$ and let $S$ be a finite set. Let 
$J^{(S)}$ be the ideal defined in \ref{not-ideal}. Let $C = A^{\otimes 
S}/J^{(S)}$, and define a map $\zeta_s$ for every $s\in S$ by composing the
natural map $A\to A^{\otimes S}$ with the quotient map. Then $(C,\zeta_s)$ is an 
$S$-closure of $A$.
\end{Pro}
\begin{proof}
We first show that for all $R\to R'$ and $a\in A\otimes_R R'$ the polynomial
\[
\Delta_a = \prod_{s\in S} \left( X - (\varepsilon_s\otimes \mathrm{Id}_{R'})(a)
\right)
\]
divides $P_a$ in $C\otimes_R R'[X]$. Let $T_a$ be the remainder of the division
of $P_a$ by $\Delta_a$ in ${(A\otimes_R R')}^{\otimes S}[X]$. Let $\varphi_a
\colon \Sym(\dual{A}) \to R'$ be the unique map such that
$\mathrm{Id}_{A}\otimes \varphi_a$ sends $\gamma$ to $a$, defined in Proposition
\ref{pro-gen-elem}.

Note that $\mathrm{Id}_{A^{\otimes S}}\otimes \varphi_a$ sends
$\varepsilon_s(\gamma)$ to $\varepsilon_s(a)$ for all $s \in S$. Hence 
$\mathrm{Id}_{A^{\otimes S}}\otimes \varphi_a(\Delta_\gamma)$ is $\Delta_a$, and
since also $\mathrm{Id}_A\otimes\varphi_a(P_\gamma)$ is $P_a$ (see Remark
\ref{rem-gen-elem-charpoly}), and quotient and remainder are unique, also
$\mathrm{Id}_{A^{\otimes S}}\otimes \varphi_a (T_\gamma)$ is $T_a$.

By definition of the ideal $J^{(S)}$ for all $i \geq 0$ the image of $t_i$ in $C
\otimes_R \Sym(\dual{A})$ via the quotient map satisfies condition \emph{2} of
Lemma \ref{lem-polys}. Hence also condition \emph{3} holds, so for all $R\to R'$
and for all $a\in A\otimes R'$ the map $\mathrm{Id}_{C} \otimes \varphi_a$ sends
$t_i$ to zero in $C\otimes_R R'$. So the coefficients of $T_a$ are zero in
$C\otimes_R R'$, and hence $P_a$ is a multiple of $\Delta_a$ in $C\otimes_R
R'[X]$, as we claimed.

We are left to show that $C$ is universal with this property. Let $B$ be an
$R$-algebra with a map $\beta_s\colon A\to B$ for each $s\in S$, and such that
for all $R\to R'$ and $a\in A\otimes_R R'$ we have that $P_a$ is a multiple of
$\prod (X - \beta_s (a))$ in $B\otimes_R R'[X]$. We need to show there exists a
unique map $C\to B$ compatible with the given maps.

By the universal property of $A^{\otimes S}$ there is a unique map
$\varphi\colon A^{\otimes S} \to B$ such that $\varphi\circ\varepsilon_s =
\beta_s$ for all $s\in S$. Since $P_\gamma$ is a multiple of $\prod (X - \beta_s
(\gamma))$ in $B\otimes_R\Sym(\dual{A})[X]$, the images of the $t_i$ in
$B\otimes_R\Sym(\dual{A})$ via $\varphi\otimes \mathrm{Id}_{\Sym(\dual{A})}$ are
zero, so they satisfy condition \emph{1} of Lemma \ref{lem-polys}, and hence
also condition \emph{2}. In particular, the map $\varphi$ is zero on the ideal
$J^{(S)}$, and hence it factors through $C$, giving the required map.

This map is necessarily unique, as $C$ is generated by the images of the
$\zeta_s$.
\end{proof} 

\begin{Rem}\label{rem-finite-set-of-generators}
Let $(\lambda_i)_{i\in I}$ be a set of generators for $\Hom_R(\Sym(\dual{A}),
R)$. Then the ideal
\[
\left\langle (\mathrm{Id}_{A^{\otimes S}}\otimes \lambda_i)(t_k)\mid k\geq 0,
i\in I \right\rangle
\]
in $A^{\otimes S}$ is equal to the ideal $J^{(S)}$ defined in Notation
\ref{not-ideal}. Moreover there exists $N\geq 0$ such that all the $t_k$ are in
$\Sym^{\leq N}(\dual{A})$ and the ideal $J^{(S)}$ can be generated by ranging
over a set of generators for $\Hom_R(\Sym^{\leq N}(\dual{A}), R)$. This gives a
finite set of generators for $J^{(S)}$.
\end{Rem}

\begin{Ex}\label{ex-S-closure-free}
Let $A$ be free of rank $n$. Let $e_1,\ldots, e_n$ be a basis of $A$ and let
$X_1,\ldots, X_n$ be the dual basis. Here we have that $A^{\otimes S}\otimes_R
\Sym(\dual{A})$ is isomorphic to $A^{\otimes S}[X_1,\ldots, X_n]$ (see also
Example \ref{ex-gen-elem}). Then the $t_i$ defined in Notation \ref{not-ideal} 
are polynomials in the $X_i$ with coefficients in $A^{\otimes S}$. For all 
monomials $X_1^{i_1}\cdots X_n^{i_n}$ we can define a linear map $R[X_1,\ldots, 
X_n]\to R$ that is one on $X_1^{i_1}\cdots X_n^{i_n}$ and zero on all other 
monomials. The images of $t_i$ via these maps are its coefficients. The ideal 
$J^{(S)}$ can then be defined by the coefficients of the $t_i$. In Section 
\ref{galois-closure-examples} we will use this set of generators for the ideal 
$J^{(S)}$ to compute examples.
\end{Ex}

We will prove as a consequence of the product formula, in Proposition
\ref{pro-Sclo-nontriv}, that (excluding trivial cases as in number \emph{1} of
Proposition \ref{pro-Sclo-special-cases}) the $S$-closure of an $R$-algebra is
not the zero algebra.

We give a list of consequences of the definition and the construction. First
some notation that we will use frequently: if $S = \{1,\ldots, m\}$ we will
denote $A^{(S)}$ by $A^{(m)}$.

\begin{Pro}\label{pro-Bharg-Sat}
Let $A$ be an $R$-algebra of rank $n$. Then $A^{(n)}$ is isomorphic to the
$S_n$-closure of Bhargava and Satriano.
\end{Pro}
\begin{proof}
The algebra $G(A/R)$ with the natural maps $f_i\colon A\to G(A/R)$ for $i =
1,\ldots,n$, has the following universal property: for every $a\in A$ the
elements $f_i(a)$ are roots of the characteristic polynomial in $G(A/R)[X]$ and
given an $R$-algebra $B$ together with maps $\beta_i\colon A\to B$ for $i =
1,\ldots, n$ such that for all $a\in A$ we have
\[
P_a(X) = \prod_i (X - \beta_i(a))
\]
there is a unique map $G(A/R) \to B$ compatible with the natural maps. Since the
construction of $G(A/R)$ commutes with base change (theorem $1$ in
\cite{BhargSat}) it also has the universal property of $A^{(n)}$.
\end{proof}

\begin{Rem}\label{rem-naive-def}
The Galois closure by Bhargava and Satriano is constructed as the quotient of
$A^{\otimes n}$ (for $A$ an $R$-algebra of rank $n$) by the ideal generated by
the differences of the coefficients of $P_a$ and of $\prod_{i} (X -
\varepsilon_i(a))$ for all $a\in A$. A similar construction for the $m$-closure
with arbitrary $m$ would be to take the quotient of $A^{\otimes m}$ by the ideal
generated by the coefficients of the remainder in the division of $P_a$ by
\[
\Delta_a = \prod_{i = 1}^m (X - \varepsilon_i(a))
\]
for all $a\in A$. The fact that $G(A/R)$ commutes with base change is surprising
and nontrivial, since we add in principle more relations if we require that
$P_a$ is a multiple of $\Delta_a$ for $a$ in any base change of $A$. We give an
example of this in Example \ref{ex-not-bcinv}, where we show that the naive
construction do not always commute with base change for $m < n$.
\end{Rem}

\begin{Pro}\label{pro-Sclo-special-cases}
Let $A$ be an $R$-algebra of rank $n$ and let $S$ be a finite set. Then:
\begin{enumerate}
\item If $\#S > n$ then $A^{(S)}$ is zero.
\item If $S=\varnothing$ then $A^{(S)}$ is $R$.
\item If $S=\{s\}$ then $A$ together with the identity map $A \to A$ is an
$S$-closure of $A$.
\end{enumerate}
\end{Pro}
\begin{proof}
$\phantom{phantom}$
\begin{enumerate}
\item Let $B$ be an $R$-algebra with maps $\beta_s\colon$ $A\to B$ such that for
all $R\to R'$ and $a\in A'$ the polynomial $P_a$ is a multiple of $D_a = \prod
(X - \beta_s(a))$ in $B'[X]$. Since both $P_a$ and $\Delta_a$ are monic, the
algebra $B$ must be $\{0\}$ because by assumption the degree of $D_a$ is
strictly bigger than the degree of $P_a$. Then $\{0\}$ has the universal
property for $A^{(S)}$.
\item Since $1$ divides every polynomial, the universal property becomes: for
every $R$-algebra $B$ there exists a unique map $A^{(S)}\to B$. Since $A^{(S)}
= R$ has this property, the proof is complete.
\item By Cayley-Hamilton (see Remark \ref{rem-Cay-Ham}) the polynomial $(X - a)$
divides $P_a(X)$ in $A[X]$ for every $R\to R'$ and any $a$ in $A'$. The same is
true for any $R$-algebra $B$ with a map $f_s\colon A\to B$. Since $f_s = f_s
\circ \mathrm{Id}_A$, we have that $(A,\Id_A)$ has the universal property of
$A^{(S)}$. The proof is complete.
\qedhere
\end{enumerate}
\end{proof}

As discussed in the introduction, given a separable field extension $L =
K[X]/(f)$ of degree $n$ we can construct a Galois closure of $L/K$ by adjoining
the $n$ roots of $f$ one by one. By definition a Galois closure of $L$ must
contain all the roots of $f$, but since the sum of the roots is equal to minus
the coefficient of $X^{n-1}$ in $f$, the field extension obtained by adding
$n-1$ roots is already a Galois closure. The next theorem shows that the same
happens for $S$-closures.

\begin{Theo}\label{theo-S-n-1}
Let $A$ be an $R$-algebra of rank $n$. Then $A^{(n-1)} \cong A^{(n)}$.
\end{Theo}
\begin{proof}
Define a map $\alpha_n\colon A\to A^{(n-1)}$ sending $a$ to $s_1(a)- \sum_i
\alpha_i (a)$ for $i =1,\ldots, n-1$. We prove that $A^{(n-1)}$ with maps
$\alpha_i$ for $i = 1,\ldots, n$ has the universal property for $A^{(n)}$.
Clearly $\alpha_n$ is linear and for all $R$-algebras $R'$ and $a\in A' =
A\otimes_R R'$, the characteristic polynomial of $a$ is equal to $\prod_i (X -
\alpha_i(a))$. From this it follows that given any $R$-algebra $B$ with $n$ maps
$\beta_i\colon A\to B$ satisfying the required property, the map $\psi \colon
A^{(n-1)} \to B$ given by the universal property of $A^{(n-1)}$ satisfies
$\beta_n = \psi\circ \alpha_n$.

We need to show that $\alpha_n$ is multiplicative. We use the following formula 
from \cite{BartCharPoly}:
\begin{equation}
s_2(a + b) = s_2(a) + s_2(b) + s_1(a)s_1(b) - s_1(ab) \label{BartS2}
\end{equation}
Since $P_a$ is equal to $\prod_i (X-\alpha_i(a))$ we have that
\[
\begin{array}{lcr}
s_1(a) = \sum_i \alpha_i(a)& \textrm{ and } &
s_2(a) = \sum_{i<j} \alpha_i(a) \alpha_j(a)
\end{array}
\]
Then we can compute
\[
\begin{split}
s_2(a + b) &= \sum_{i<j} \alpha_i(a+b)\alpha_j(a+b) = \\
&= \sum_{i<j}\alpha_i(a)\alpha_j(a) + \sum_{i<j}\alpha_i(b)\alpha_j(b) +
\sum_{i\neq j}\alpha_i(a)\alpha_j(b)= \\
&= s_2(a) + s_2(b) + \sum_{i\neq j} \alpha_i(a)\alpha_j(b)
\end{split}
\]
and comparing with formula \eqref{BartS2}:
\[
\sum_i \alpha_i(ab) = s_1(ab) = s_1(a)s_1(b) - \sum_{i\neq j}
\alpha_i(a)\alpha_j(b) = \sum_i \alpha_i(a)\alpha_i(b).
\]
Since $\alpha_i$ is multiplicative for $i = 1,\ldots, n-1$ follows that also
$\alpha_n$ is multiplicative, as we wanted to show.
\end{proof}

\begin{Cor} \label{galois-closure-rank-two}
Let $A$ be an $R$-algebra of rank $2$. Then $A$ with the identity and the
natural involution $a \mapsto s_1(a) - a$, is a $2$-closure of $A$.
\end{Cor}
\begin{proof}
Follows from Theorem \ref{theo-S-n-1} and from number \emph{3} of Proposition
\ref{pro-Sclo-special-cases}.
\end{proof}

\begin{Rem}\label{rem-Sclo-action}
For any $R$-algebra $A$ and any finite set $S$, by the universal property of
$A^{(S)}$, there is a natural action of the symmetric group of $S$ on $A^{(S)}$,
exchanging the natural maps. This action will be discussed in detail in a
future paper.
\end{Rem}


\section{The product formula}\label{sec-product-formula}

The theorem we are going to prove is a generalization to $S$-closures of the
product formula that is proved in \cite{BhargSat}; it is a formula to compute
the $S$-closure of a product of algebras in terms of closures of the factors. We
state the theorem here.

\begin{InTheo}[Theorem \ref{theo-prod-formula}]
Let $m$ be a positive integer. For $i = 1, \ldots, m$ let $A_i$ be an
$R$-algebra of rank $n_i$. Let $A$ be $A_1\times \cdots \times A_m$, an
$R$-algebra of rank $n= \sum n_i$. Let $S$ be a finite set and let $\mathscr{F}$
be the set of all maps $S\to\{1,\ldots,m\}$. Fix $F\in \mathscr{F}$ and let $S_i
= F^{-1}(i)$. Write
\[
A^{(F)} = \bigotimes_{i =1}^{m} A_i^{(S_i)}
\]
and let $\alpha_{s,i}$ for $s\in S_i$ be the natural map $A_i \to A_i^{(S_i)}$.
Define an $R$-algebra
\[
C = \prod_{F\in \mathscr{F}} A^{(F)}
\]
and maps $\delta_s \colon A \to C$ for $s \in S$ by
\[
(\delta_s(a_1,\ldots, a_m))_F = 1\otimes \cdots\otimes \alpha_{s,i}(a_i)\otimes
\cdots\otimes 1, \textrm{ with } i = F(s).
\]
Then $(C, (\delta_s)_{s\in S})$ is the $S$-closure of $A/R$.
\end{InTheo}

We will give the proof of the theorem after proving some lemmas.

\begin{Lem}\label{lem-coprime}
Let $R$ be a ring and let $t$ be a positive integer. Then there exists
polynomials $u(X), v(X) \in R[X]$ such that
\[
1 = u(X)X^t + v(X)(X-1)^t.
\]
\end{Lem}
\begin{proof}
Let $I$ be the ideal generated by $X^t$ and $(X-1)^t$. We show that the quotient
ring $S = R[X] / I$ is trivial, so that $1\in I$. The image of $X$ in $S$ is
nilpotent since $X^t = 0$, so $X-1$ is a unit in $S$. But $X-1$ is also
nilpotent, hence $S$ is trivial, as we wanted to show.
\end{proof}

Let $P$ and $Q$ be in $R[X]$. We will write $P \mid Q$ for ``$P$ divides $Q$''.

\begin{Lem}\label{lem-charpoly-prod}
Let $A_i$ be $R$-algebras with $A_i$ of rank $n_i$, for $i = 1,\ldots, m$. Let
$S_1,\ldots, S_m$ be finite sets and let $B$ be an $R$-algebra with maps
$j_{s,i}\colon A_i \to B$ for $i = 1,\ldots, m$ and $s\in S_i$. Then the
following are equivalent:
\begin{enumerate}
\item For all $a = (a_1,\ldots, a_m)\in A$ we have
\[
\prod_{i=1}^{m}\left(\prod_{s\in S_i}(X - j_{s,i}(a_i))\right) \mid P_{a}(X)
\]
in $B[X]$, where $P_a$ is the characteristic polynomial of $a$.
\item For all $i\in \{1,\ldots, m\}$ and for all $a_i \in A_i$ we have
\[
\prod_{s\in S_i}(X - j_{s,i}(a_i))\mid P_{a_i}(X)
\]
in $B[X]$, where $P_{a_i}$ is the characteristic polynomial of $a_i$ in $A_i$.
\end{enumerate}
\end{Lem}
\begin{proof}
Note that $P_a$ is equal to $\prod_i P_{a_i}(X)$, so clearly the second
condition implies the first. Assume the first condition holds and fix $i \in
\{1,\ldots, m\}$. Setting $a_j = 0$ for $j\neq i$ we have:
\[
\prod_{s\in S_i}(X-j_{s,i}(a_i)) X^{N_1} \mid P_{a_i}(X)X^{N_2}
\]
with $N_1 = \sum_{i\neq j} \# S_j$ and $N_2 = \sum_{i\neq j} n_j$. Setting
$a_j = 1$ for $j \neq i$ we have
\[
\prod_{s\in S_i}(X-j_{s,i}(a_i)) (X - 1)^{N_1} \mid P_{a_i}(X)(X-1)^{N_2}.
\]
Put $t = N_2 - N_1$, and note that if $t<0$ then
\[
P_{a_i}(X) = X^{-t} \prod_{s\in S_i}(X - j_{s,i}(a_i))
\]
and so the second condition holds. Assume $t\geq 0$, there exist $f(X)$ and
$g(X)$ in $B'[X]$ such that:
\begin{equation}
P_{a_i}(X)X^{t} = \prod_{s\in S_i}(X-j_{s,i}(a_i)) f(X)
\label{lem-charpoly-1}
\end{equation}
\begin{equation}
P_{a_i}(X)(X - 1)^{t} = \prod_{s\in S_i}(X-j_{s,i}(a_i)) g(X)
\label{lem-charpoly-2}
\end{equation}
By Lemma \ref{lem-coprime} there exist $u(X)$ and $v(X)$ in $B'[X]$ such that
$u(X) X^t + v(X) (X-1)^t = 1$. Multiplying \eqref{lem-charpoly-1} by $u(X)$ and
\eqref{lem-charpoly-2} by $v(X)$ and adding the results we get
\[
P_{a_i}(X) = (u(X) f(X) + v(X) g(X))\prod_{s\in S_i}(X - j_{s,i}(a_i))
\]
so the second condition holds and the proof is complete.
\end{proof}

\begin{Lem}\label{lem-connected-rings}
Let $A=\prod_{i=1}^m A_i$ be a finite product of rings. If $B$ is a connected
ring then given a morphism $f\colon A\to B$ there exists a unique index $i$ and
a unique morphism $g\colon A_i \to B$ such that $f = g \circ \pi_i$, where
$\pi_i\colon A\to A_i$ is the natural projection.
\[
\xymatrix{
A \ar[r]^f  \ar[d]_{\pi_i} & B \\
A_i \ar@{..>}[ur]_g
}
\]
\end{Lem}
\begin{proof}
Since $B$ is connected the only idempotents are $0$ and $1$. Let $1_i$ be the
element $(0,\ldots, 1, \ldots, 0)$ of $A$, where $1$ is in the $i-th$ position.
Since $1_i$ is idempotent $f(1_i)$ is idempotent in $B$. If $f(1_i)= 0$ for all
$i$, then $f$ would send $1$ to $0$, so there exist at least one $i_0$ such that
$f(1_{i_0}) = 1$. For $j\neq i_0$, we have
\[
0 = f(0) = f(1_{i_0} 1_j) = f(1_{i_0}) f(1_j) = f(1_j)
\]
so $1_j$ is zero if $j\neq i_0$ and hence $i_0$ is unique. Then $f$ factors
through a unique $A_{i_0}$, as we wanted to show.
\end{proof}

We now prove the product formula.

\begin{Theo}\label{theo-prod-formula}
Let $m$ be a positive integer. For $i = 1, \ldots, m$ let $A_i$ be an
$R$-algebra of rank $n_i$. Let $A$ be $A_1\times \cdots \times A_m$, an
$R$-algebra of rank $n= \sum n_i$. Let $S$ be a finite set and let $\mathscr{F}$
be the set of all maps $S\to\{1,\ldots,m\}$. Fix $F\in \mathscr{F}$ and let $S_i
= F^{-1}(i)$. Write
\[
A^{(F)} = \bigotimes_{i =1}^{m} A_i^{(S_i)}
\]
and let $\alpha_{s,i}$ for $s\in S_i$ be the natural map $A_i \to A_i^{(S_i)}$.
Define an $R$-algebra
\[
C = \prod_{F\in \mathscr{F}} A^{(F)}
\]
and maps $\delta_s \colon A \to C$ for $s \in S$ by
\[
(\delta_s(a_1,\ldots, a_m))_F = 1\otimes \cdots\otimes \alpha_{s,i}(a_i)\otimes
\cdots\otimes 1, \textrm{ with } i = F(s).
\]
Then $(C, (\delta_s)_{s\in S})$ is the $S$-closure of $A/R$.
\end{Theo}
\begin{proof}
We give first a summary of the proof: we show that for every $R\to R'$ and every
$a\in A'= A\otimes_R R'$ the characteristic polynomial $P_a(X)$ is a multiple of
\[
\Delta_a(X) = \prod_{s\in S} (X - \delta_s(a))
\]
in $C\otimes_R R'[X]$; then we show that for every \emph{connected} $R$-algebra
$B$ given with a map $\beta_s\colon A\to B$ for each $s\in S$, such that for
every $R\to R'$ and every $a\in A'$ the characteristic polynomial $P_a(X)$ is a
multiple of $\prod (X - \beta_s(a))$ in $B\otimes_R R'[X]$ there is a unique map
$C\to B$ commuting with all the natural maps; from this we deduce the theorem
for $R$ a finitely generated $\mathbb{Z}$-algebra; then we prove the theorem in
general.

That $P_a$ is a multiple of $\Delta_a$ in $C\otimes_R R'[X]$ for every $R\to R'$
and for every $a\in A'$ follows from the easy implication in Lemma
\ref{lem-charpoly-prod} and the defining property of $A^{(S_i)}$.

Suppose $B$ is a \emph{connected} $R$-algebra with maps $\beta_s\colon A\to B$
for $s\in S$ and such that for all $R\to R'$ and for all $a\in A'$ we have
\[
\prod_{s\in S}(X - \beta_s(a)) \mid P_{a}(X)
\]
in $B\otimes_R R'[x]$. We show that there exists a unique map $\varphi\colon
C\to B$ such that for all $s\in S$ we have $\beta_s = \varphi \circ \delta_s$.
Since $B$ is connected, by Lemma \ref{lem-connected-rings} we can define a map
$F\colon S \to \{1,\ldots, m\}$ by setting $F(s) = i$ if and only if $\beta_s$
factors through $A_i$, i.e. there exists a map $\beta_{s,i}$ such that $\beta_s
= \beta_{s,i} \circ \pi_i$, where $\pi_i$ is the projection $A\to A_i$. For all
$a\in A'$ we have that
\[
\prod_{s\in S}(X - \beta_s(a)) = \prod_{i=1,\ldots, m}\left( \prod_{s\in
F^{-1}(i)}\left( X - \beta_{s,i}(a_i) \right)\right)
\]
holds in $B'[x]$. By Lemma \ref{lem-charpoly-prod} for all $i \in \{1,\ldots,
m\}$ and for all $a_i\in A_i'$ we have then
\[
\prod_{s\in F^{-1}(i)}(X - \beta_{s,i}(a_i)) \mid P_{a_i}(X)
\]
so the universal property of $A_i^{(S_i)}$ gives a unique map $f_i\colon
A_i^{(S_i)} \to B$ such that for all $s\in S_i$ we have $\beta_{s,i} = f_i\circ
\alpha_{s,i}$. Tensoring these maps, we get a unique map $f\colon A^{(F)}\to B$
such that for all $i$ we have $f_i = f\circ \varepsilon_i$ where $\varepsilon_i$
is the natural map $A_i^{(S_i)} \to A^{(F)}$. Composing with the projection
$\pi_F\colon C \to A^{(F)}$ we obtain a map $\varphi\colon C\to B$. The
following commutative diagram summarizes the situation for $s\in S_i$ fixed.
\[
\xymatrix{
A_i^{(S_i)} \ar[r]^-{\varepsilon_i} \ar@{..>}[dr]^-{f_i}&
A^{(F)} \ar@{..>}[d]^-{f}&
C \ar[l]_-{\pi_F} \ar@{..>}[dl]^-{\varphi}\\
A_i \ar[u]^-{\alpha_{s,i}} \ar[r]^-{\beta_{s,i}}&
B
}
\]
We have: $\pi_F\circ\delta_s = \varepsilon_{i}\circ \alpha_{s,{i}}\circ\pi_{i}$.
So for every $s\in S$ we have that $\beta_s$ is 
\[
\beta_{s,i}\circ\pi_i=f\circ\varepsilon_i \circ \alpha_{s,i}\circ\pi_i
=f\circ \pi_F\circ\delta_s =\varphi \circ\delta_s
\]
so $\varphi$ satisfies the required condition. Uniqueness follows because the
images of the $\delta_s$ generate $C$.

Suppose now $R$ is finitely generated over $\mathbb{Z}$. From the construction
of $A^{(S)}$ in Proposition \ref{pro-gal-constr} we see that $A^{(S)}$ is
finitely generated over $R$ and hence also over $\mathbb{Z}$, so it is
noetherian. Every noetherian ring is a finite product of connected rings by
\cite[Chapter $2$, Exercise $2.13$ (c)]{Hartshorne} and \cite[Chapter $1$,
Proposition $1.5$]{Hartshorne}. Write $A^{(S)}$ as $\prod_{j\in J}R_j$, with $J$ 
a finite set and $R_j$ connected. By applying the projection map $\pi_j\colon 
A^{(S)}\to R_j$ we get that
\[
\prod_{s\in S}(X - \pi_j\circ\delta_s(a)) \mid P_{a}(X)
\]
holds in every $R_j'[X]$, and since the $R_j$ are connected the previous
argument gives a map $\varphi_j\colon C\to B_j$ such that for all $s\in S$ we
have $\pi_j\circ \alpha_s = \varphi_j \circ\delta_s$. By the universal property
of the product we get a map $\varphi\colon C\to A^{(S)}$ such that for all $j\in
J$ we have $\varphi_j = \pi_j\circ\varphi$. The following diagram shows the
situation for $j$ and $s$ fixed.
\[
\xymatrix{
A \ar[d]_-{\alpha_s} \ar[r]^-{\delta_s}&
A^{(S)} \ar[d]^-{\pi_j} \\
C \ar[r]_-{\varphi_j} \ar@{..>}[ur]^-{\varphi}&
R_j
}
\]
For all $s\in S$ and for all $j\in J$ we have $\pi_j\circ\delta_s = \pi_j\circ
\varphi\circ\delta_s$, and hence the universal property of the product gives
$\varphi \circ \delta_s = \delta_s$. So there exists a map $C\to A^{(S)}$
commuting with the natural maps and this is sufficient to conclude for $R$
finitely generated over $\mathbb{Z}$.

Back to the general case: we no longer assume $R$ to be finitely generated over
$\mathbb{Z}$. The $R$-algebra $A$ is finitely presented over $R$, so there
exists a subring $R_0$ of $R$, finitely generated over $\mathbb{Z}$, and a
finite $R_0$-algebra $A_0$ such that $A\cong R\otimes_{R_0} A_0$. Define the
$R_0$-algebra $C_0$ in the obvious way and note that $C\cong R\otimes_{R_0}
C_0$. We proved $A_0^{(S)}$ and $C_0$ are isomorphic and the constructions of
$A^{(S)}$ and of $C$ both commute with base change, so $A^{(S)}$ and $C$ are
isomorphic and the proof is complete.
\end{proof}

The product formula in \cite{BhargSat} is now a corollary of Theorem
\ref{theo-prod-formula}.

\begin{Cor}\label{cor-formula-nclo}
For $i = 1, \ldots, m$ let $A_i$ be an $R$-algebra of rank $n_i$. Let $A$ be
the product of the $A_i$, an $R$-algebra of rank $n= \sum n_i$. Then
the Galois closure of $A$ satisfies
\[
A^{(n)}\cong \left(\bigotimes_{i=1}^m A_i^{(n_i)}\right)^{\frac{n!}{n_1!\,\cdots
\,n_m!}}
\]
\end{Cor}
\begin{proof}
By Theorem \ref{theo-prod-formula} we can write $A^{(n)}$ as a product indexed
over all maps $F\colon \{1,\ldots, n\} \to \{1,\ldots, m\}$. By Proposition
\ref{pro-Sclo-special-cases} a factor is not zero if and only if for all $i$ we
have $\#F^{-1}(i) = n_i$, so there are $\frac{n}{n_1!\,\cdots \,n_m!}$ non-zero
factors and they are all isomorphic to ${\bigotimes}_{i} A_i^{(n_i)}$. The
statement then follows from Proposition \ref{pro-Bharg-Sat}.
\end{proof}

Using the product formula we can compute the $S$-closure of an \'etale 
$R$-algebra of rank $n$, generalizing theorem $4$ in \cite{BhargSat}. We first 
present some basic facts on finite \'etale algebras. For what we need, the 
presentation in \cite{GaloisSchemes} is more than sufficient.

\begin{Def}\label{def-finite-etale}
Let $R$ be a ring and let $A$ be an $R$-algebra of rank $n$. Recall that 
$s_1\colon A\to R$ denotes the trace map. We say $A$ is \emph{finite \'etale} 
over $R$ if the map $A\to \dual{A}$ given by
\[
a \mapsto ( b \mapsto s_1(ab) )
\]
is an isomorphism.
\end{Def}

\begin{Pro}\label{etale-equivalent-def}
Let $R$ be a connected ring and let $A$ be an $R$-algebra. Then $A$ is finite
\'etale if and only if there exists a finite projective $R$-algebra $R'$, with
$R\to R'$ injective, such that $A\otimes_R R'$ is isomorphic to $(R')^n$ as an
$R$-algebra for some $n\geq 0$.
\end{Pro}
\begin{proof}
This is standard, for example follows from \cite[Theorem $5.10$]{GaloisSchemes}
by translating the conditions for schemes to the affine case.
\end{proof}

\begin{Def}
Given a profinite group $G$ the \emph{category of finite $G$-sets} is the
category whose objects are finite sets equipped with a continuous action of $G$,
and morphisms are maps compatible with the action.
\end{Def}

\begin{Theo}\label{theorem-etale-equivalence}
Let $R$ be a connected ring and let $K$ be a separably closed field. Let
$\alpha\colon R\to K$ be a ring homomorphism. Then there exist
\begin{enumerate}
\item A profinite group $\pi = \pi(R,\alpha)$.
\item An equivalence of categories \[F\colon \{\textrm{finite \'etale }
R\textrm{-algebras}\}^{\mathrm{op}} \to \{\textrm{finite }
\pi\textrm{-sets}\}.\]
\item An isomorphism of functors from $\Hom_{R\textrm{-}\mathrm{Alg}}(-, K)$ to
the composition of the forgetful functor $\{\textrm{finite }
\pi\textrm{-sets}\}\to \{\textrm{sets}\}$ with $F$.
\end{enumerate}
Moreover, we have:
\begin{itemize}
\item[{a}.] The group $\pi$ is uniquely determined up to isomorphism.
\item[{b}.] For all finite sets $T$ we have $F(R^T) = T$ with trivial action of
$\pi$.
\item[{c}.] Given finite \'etale $R$-algebras $A$ and $B$ the tensor product
$A\otimes_R B$ is \'etale and $F(A\otimes_R B) = F(A) \times F(B)$ with the
induced action.
\end{itemize}
\end{Theo}
\begin{proof}
This is a combination of standard results on finite \'etale algebras, and on
equivalences of categories. Results in \cite[Section $5$]{GaloisSchemes} contain
everything that is needed.
\end{proof}

\begin{Def}
The group $\pi$ in Theorem \ref{theorem-etale-equivalence} is called the 
\emph{\'etale fundamental group of $R$ in $\alpha$}.
\end{Def}

\begin{Pro}\label{Sclo-etale}
Let $R$ be a connected ring and $K$ a separably closed field. Let $\alpha\colon
R\to K$ be a ring homomorphism and let $\pi$ be the \'etale fundamental group
of $R$ in $\alpha$. Let $A$ be a finite \'etale $R$-algebra corresponding to a
$\pi$-set $T$. Then for any finite set $S$ the algebra $A^{(S)}$ is finite
\'etale and corresponds to the set $I$ of injective maps $S\to T$ with the
action of $\pi$ induced by the one on $T$.
\end{Pro}
\begin{proof}
First suppose the action of $\pi$ on $T$ is trivial so that $A = R^T$. In this
case by Theorem \ref{theo-prod-formula} the $S$-closure of $A$ is a product
indexed over all maps $F\colon S\to T$ of the $F^{-1}(t)$-closure of $R$, for
$t$ in $T$. If $F$ is not injective then there exists a $t\in T$ such that
$F^{-1}(t)$ has at least two elements so the $F^{-1}(t)$-closure of $R$ is zero
by number \emph{1} of Proposition \ref{pro-Sclo-special-cases}. If $F$ is
injective, then for all $t\in T$ the $F^{-1}(t)$-closure of $R$ is $R$, by
number \emph{3} of Proposition \ref{pro-Sclo-special-cases}. Hence the
$S$-closure of $A$ is $R^{I}$.

In general, by Proposition \ref{etale-equivalent-def} there exists a finite
projective $R$-algebra $R'$ with $R\to R'$ injective such that $A\otimes_R R'$
is isomorphic to $(R')^n$ for some $n\geq 0$. We proved the $S$-closure of
$(R')^n$ is isomorphic to $(R')^N$ for some $N$, then by Proposition
\ref{etale-equivalent-def} the $S$-closure of $A$ is finite \'etale. The
$\pi$-set corresponding to $A^{(S)}$ is $\Hom_R(A^{(S)}, K)$, which is
isomorphic to $\Hom_K((A\otimes_R K)^{(S)}, K)$ as $\pi$-sets. We are then
reduced to proving that for $A$ an \'etale algebra over $K$, corresponding to a
set $T$, the $S$-closure of $A$ corresponds to $I$. Since in this case $A = K^T$
this was proven already.
\end{proof}

\begin{Rem}\label{rem-expected-rank}
Note that Proposition \ref{Sclo-etale} implies that if $A$ is an \'etale
$R$-algebra of rank $n$ then $A^{(S)}$ is \'etale of rank $n (n - 1) \cdots (n -
\#S +1)$. This will be called the \emph{expected rank} of $A^{(S)}$. For general
algebras over fields it is possible that the rank of the $S$-closure is not the
expected one. We will give examples in Section \ref{galois-closure-examples}.
\end{Rem}

An important consequence of the product formula is Proposition
\ref{pro-Sclo-nontriv}, which says that the $S$-closure is not zero, excluding
trivial cases. This was not previously known for the construction of Bhargava
and Satriano. We need some facts about algebras over fields, which will be used
also later on.

\begin{Lem}\label{lem-finite-algebra-over-field}
Let $K$ be a field and let $A$ be a finite $K$-algebra. Then $A$ is isomorphic
to a finite product of $K$-algebras $\prod_i A_i$ with $A_i$ local with
nilpotent maximal ideal.
\end{Lem}
\begin{proof}
Follows from \cite[Corollary $2.15$, page $76$]{EisenbudComm}. It can also be
proved directly as in \cite[Theorem $2.6$]{GaloisSchemes}.
\end{proof}

\begin{Cor}\label{lem-local-over-alg-closed-has-unique-map}
Let $K$ be an algebraically closed field and let $A$ be a finite $K$-algebra.
Then $A$ is local if and only if there exists a unique $K$-algebra map $A\to K$.
\end{Cor}
\begin{proof}
Suppose $A$ is local with maximal ideal $\mathfrak{m}$. Since $A$ is finite over
$K$ the residue field is an algebraic extension of $K$. Since $K$ is
algebraically closed the residue field is $K$. In particular, there exists a
$K$-algebra map $\pi\colon A\to K$. Since $A = K \oplus \mathfrak{m}$, and the
kernel of any map $A\to K$ is $\mathfrak{m}$, all those maps are equal to
$\pi$. So if $A$ is local $\pi$ is the unique map $A\to K$.

By Lemma \ref{lem-finite-algebra-over-field} any $K$-algebra $A$ is a finite
product of local algebras $A_i$. For each $A_i$ we have a map $A \to A_i\to K$,
and different factors give different maps. So if there is a unique map $A\to K$
there is one factor in the product, and hence $A$ is local.
\end{proof}

\begin{Lem}\label{cor-Sclo-local}
Let $K$ be an algebraically closed field and let $A$ be a connected $K$-algebra
of rank $n$. Let $S$ be a set with $\#S\leq n$. Then $A^{(S)}$ is local.
\end{Lem}
\begin{proof}
By Lemma \ref{lem-finite-algebra-over-field} and Lemma
\ref{lem-local-over-alg-closed-has-unique-map}, we have that $A$ is local with
nilpotent maximal ideal $\mathfrak{m}$ and residue field $K$. Let $f\colon
A\to K$ be the quotient map. Any element $a\in A$ can be written as $r + m$ with
$m \in \mathfrak{m}$ and $r \in K$. With this notation one has $f(a) = r$ and
the characteristic polynomial of $a$ is $(X - r)^n$. Since this is a multiple of
$(X - f(a))^{\# S}$, the universal property of $A^{(S)}$ gives a unique
$K$-algebra map $\varphi\colon A^{(S)}\to K$ such that for all $s\in S$ we have
$\varphi\circ\delta_s = f$.

Any other $K$-algebra map $A^{(S)}\to K$ must also satisfy the same condition,
and hence coincide with $\varphi$. So $A^{(S)}$ is local by Lemma
\ref{lem-local-over-alg-closed-has-unique-map}.
\end{proof}

\begin{Pro}\label{pro-Sclo-nontriv}
Let $R$ be a non-zero ring and $A$ an $R$-algebra of rank $n>0$. Let $S$ be a
set with $\#S \leq n$. Then $A^{(S)}$ is not zero.
\end{Pro}
\begin{proof}
Suppose first that $A$ is connected over an algebraically closed field. By
Corollary \ref{cor-Sclo-local} in this case $A^{(S)}$ is local and hence not
zero.

If $A$ is any algebra of rank $n$ over an algebraically closed field, then by
Lemma \ref{lem-finite-algebra-over-field}, we have that $A$ is a finite product
of connected $K$-algebras, and by the above and the assumption on $\#S$, not all
the $A^{(F)}$ in the formula in Thoerem \ref{theo-prod-formula} are zero, so
$A^{(S)}$ is not zero.

For the general case given $R\to K$ with $K$ an algebraically closed field the
non-zero algebra $(A\otimes_R K)^{(S)}$ is isomorphic to $A^{(S)}\otimes_R K$
since the $S$-closure commutes with base change. So $A^{(S)}$ is not zero, as we
wanted to show.
\end{proof}


\section{Polynomial laws}\label{sec-furt-gen}

Given an $R$-algebra $A$ of rank $n$ and a sequence $n_1,\ldots, n_t$ with
$\sum_i n_i = n$, Daniel Ferrand in \cite{FonctNorme} constructs an $R$-algebra
$P^{(n_1,\ldots, n_t)}(A)$, using norms. In this section we will show that
his definition is equivalent to the following: for all $R\to R'$ and all $a\in
A\otimes_R R'$ the characteristic polynomial of $a$ splits as a product of $t$
polynomials of degrees $n_1,\ldots, n_t$ in $P^{(n_1,\ldots, n_t)}(A)\otimes_R
R'[X]$, and $P^{(n_1,\ldots, n_t)}(A)$ is universal with this property.  We will
make this precise in Proposition \ref{pro-def-pa} and we will show in
Proposition \ref{pro-poly-law-rel-with-m-closure} that $P^{(1, \ldots, 1,
n-m)}(A)$ (with $m$ ones) is isomorphic to the $m$-closure of $A$.

The precise formulation of these results uses polynomial laws. This tool was
first introduced by Norbert Roby in \cite{PolMaps}; Daniel Ferrand in
\cite{FonctNorme} gives a very clear presentation of the topic, though not as
complete as the one by Roby. Here we will just give the basic definitions and 
properties we need.

In \cite[Lemme $4.1.1$]{FonctNorme} Daniel Ferrand also proves a product formula
that generalizes to $P^{(n_1,\ldots, n_t)}(A)$ the one we proved in Theorem
\ref{theo-prod-formula}. We will not give a proof of this formula here.

If $A$ is the algebra $R[x]/(f)$ for some monic polynomial $f$, a construction
for $P^{(n_1,\ldots, n_t)}(A)$ is given by Dan Laksov in \cite{FactAlg}. It is
probably possible to generalize the construction by Laksov to define
$P^{(n_1,\ldots, n_t)}(A)$ in a way similar to our definition of $A^{(m)}$,
using generic elements, but we will not do it here.

\begin{Def}
Let $R$ be a ring and let $M$ be an $R$-module. We denote by $\underline{M}$
the functor $R\textrm{-Alg} \to \textrm{Set}$ sending an $R$-algebra $S$ to the
set $M\otimes_R S$ and a morphism $f\colon S\to S'$ to $\mathrm{Id}_M\otimes f$.
\end{Def}

\begin{Def}
Let $R$ be a ring and let $M$ and $N$ be two $R$-modules. A \emph{polynomial
law} is a natural transformation $f\colon \underline{M} \to \underline{N}$.

In detail: a polynomial law $\underline{M}\to \underline{N}$ is given by a
(set-theoretical) map $f_{S}\colon M\otimes_R S \to N\otimes_R S$ for each
$R$-algebra $S$ such that for all $g\colon S \to S'$ the following diagram
commutes:
\[
\xymatrix{
M\otimes_R S  \ar[r]^{f_S} \ar[d]_{\mathrm{Id}_M\otimes g}&
N\otimes_R S  \ar[d]^{\mathrm{Id}_N\otimes g}\\
M\otimes_R S' \ar[r]_{f_{S'}}&
N\otimes_R S'
}
\]
\end{Def}

\begin{Rem}
Let $R$ be a ring. A polynomial law $\underline{R^m}\to\underline{R^n}$ is a
morphism $\mathbb{A}^m_R\to\mathbb{A}^n_R$.
\end{Rem}

\begin{Def}
Let $R$ be a ring and $M, N$ be $R$-modules. A polynomial law $f\colon
\underline{M}\to \underline{N}$ is called \emph{homogeneous} of degree $n$ if
for all $R$-algebras $S$, all elements $x\in M\otimes_R S$ and all elements
$s\in S$ we have:
\[
f_S(sx) = s^n f_S(x).
\]
\end{Def}

\begin{Ex}
Any polynomial law homogeneous of degree zero comes from a constant map $M\to
N$. Any polynomial law homogeneous of degree one comes from a linear map $M\to
N$. This is somehow surprising, since no additivity is required. A proof can be
found in \cite[Chapitre I, \S $11$]{PolMaps}.
\end{Ex}

\begin{Rem}
In this section a linear map $\alpha\colon M\to N$ will be seen as a polynomial
law $\underline{M}\to \underline{N}$ homogeneous of degree $1$. In particular
given an $R$-algebra $S$ the notation $\alpha_S$ will be used instead of
$\alpha\otimes \mathrm{Id}_S$.
\end{Rem}

\begin{Def}
Let $R$ be a ring and let $A, B$ be $R$-algebras. A polynomial law $f\colon
\underline{A}\to\underline{B}$ is called \emph{multiplicative} if for all
$R$-algebras $S$ and for all elements $x, y \in A\otimes_R S$ we have:
\[
f(xy) = f(x)f(y) \quad\textrm{ and }\quad f(1) = 1
\]
\end{Def}

\begin{Ex}
Let $M$ be a finitely generated projective $R$-module. The map $s_i$ defined in 
Section \ref{galois-closure-preliminaries} extends to a homogeneous polynomial 
law of degree $i$, for all $i$. If $A$ is an $R$-algebra of rank $n$, then the 
norm law $s_n\colon \underline{A}\to \underline{R}$ is moreover multiplicative.
\end{Ex}

The following proposition is in Roby (See \cite[Chapitre I, \S $12$]{PolMaps}),
but only in the case when $M$ is free.

\begin{Pro}\label{lemma-polynomial-law-generic-element}
Let $R$ be a ring, $M$ a projective finitely generated $R$-module, and $N$ be
any $R$-module. Denote by $S$ the ring $\Sym(\dual{M})$. Then for every element
$\eta$ in $N\otimes_R S$ there exists a unique polynomial law $f\colon
\underline{M}\to \underline{N}$ such that $f_{S}(\gamma) = \eta$, with $\gamma
\in M\otimes_R S$ the generic element of $M$.
\end{Pro}
\begin{proof}
We prove uniqueness first. Let $f\colon \underline{M}\to \underline{N}$ be a
polynomial law. Let $R'$ be any $R$-algebra and let $x\in M' = M\otimes_R R'$.
We show that knowing $f_S(\gamma)$ we can determine $f_{R'}(x)$. By Proposition
\ref{pro-gen-elem} there exists a unique $R$-algebra map $\varphi_x:\colon S\to
R'$ such that $\mathrm{Id}_M \otimes \varphi_x(\gamma) = x$. By definition of
polynomial law the following diagram commutes
\[
\xymatrix @C 3.5em{
M\otimes S \ar[r]^-{\mathrm{Id}_M\otimes\varphi_x} \ar[d]_{f_{S}} & 
M' \ar[d]^{f_{R'}} \\
N\otimes S \ar[r]^-{\mathrm{Id}_N\otimes\varphi_x} &
N'
}
\]
Then we have $f_{R'}(x) = (\mathrm{Id}_M\otimes \varphi_x)(f_S(\gamma))$. So
the claim is proved.

Now for existence, let $\eta$ be in $N\otimes_R S$. For any $R$-algebra $R'$
and any $x\in M'$ we define $f(x) = \mathrm{Id}_N\otimes\varphi_x(\eta)$. We
need to check this is a polynomial law, i.e. that, given two $R$-algebras $A$
and $B$, and a map $\psi\colon A\to B$ the following diagram commutes:
\[
\xymatrix @C 3.5em{
M \otimes_R A \ar[r]^{\mathrm{Id}_M\otimes\psi}\ar[d]_{f_{A}} &
M\otimes B \ar[d]^{f_B}\\
N\otimes A \ar[r]_{\mathrm{Id}_N\otimes\psi} & N\otimes B
}
\]
Fix $x$ in $M\otimes_R A$ and let $y$ be $\mathrm{Id}_M\otimes \psi(x)$ in
$M\otimes_R B$. The polynomial laws $\psi \circ \varphi_x$ and $\varphi_{y}$
both have value $y$ on $\gamma$, hence they are equal by the above. So we can
write:
\[
(\mathrm{Id}_N\otimes\psi)(f_A(x)) =
(\mathrm{Id}_N\otimes\psi)(\mathrm{Id}_N\otimes\varphi_x(\eta)) =
(\mathrm{Id}_N\otimes\varphi_y)(\eta) = f_B(y)
\]
and since $f_B(y) = f_B(\mathrm{Id}_M\otimes\psi(x))$, the diagram commutes as
we wanted to show.
\end{proof}

\begin{Rem}\label{rem-ring-poly-laws}
If $M$ is a projective finitely generated module, Lemma
\ref{lemma-polynomial-law-generic-element} gives an isomorphism 
\[
\mathscr{P}(M,N) \cong N\otimes_R \Sym(\dual{M}),
\]
in particular $\mathscr{P}(M,N)$ is in this case a graded module. If moreover
$B$ is an $R$-algebra, then $\mathscr{P}(M,B)$ has a structure of graded
$R$-algebra, and it is isomorphic to $B\otimes_R \Sym(\dual{M})$ as a graded
$R$-algebras.
\end{Rem}

Recall that given $R$-algebras $R\to A$ and $R\to S$, with $A$ finite locally
free, and an element $a$ in $A\otimes_R S$, we denote by $P_a\in S[X]$ the
characteristic polynomial of $a$. If $f$ is the endomorphism of $A\otimes_R S$
given by multiplication by $a$ this is defined as the determinant of the
endomorphism $( \Id \otimes X - f\otimes \Id )$. Equivalently, we can write it
as the image of $(X - a)\in A\otimes_R S[X]$ via the polynomial law $s_n$. We
can now give the definition of $A^{(n_1,\ldots, n_t)}$.

\begin{Pro}\label{pro-def-pa}
Let $A$ be an $R$-algebra of rank $n$. Let $C$ be an $R$-algebra and let 
$\delta_i\colon \underline{A}\to \underline{C}$ for $i = 1,\ldots, t$ be
polynomial laws. Suppose $\delta_i$ is homogeneous of degree $n_i$ and that
$\sum n_i$ is $n$. Then the following are equivalent:
\begin{enumerate}
\item The equality $s_n = \prod \delta_i$ holds (as polynomial laws) and
$(C,(\delta_i)_{i})$ is universal with this property.
\item For all $R\to S$ and for all $a\in A\otimes_R S$ we have
\[
P_a(X) = \prod_i \delta_{i, S[X]}(X - a)
\]
in $C\otimes_R S[X]$ and $(C,(\delta_i)_i)$ is universal with this property.
\item Let $\eta_i$ be the element of $C\otimes_R \Sym(\dual{A})$ corresponding
to $\delta_i$ and let $\mathrm{det}$ be the one corresponding to $s_n$. Then
$\mathrm{det}$ is equal to the product of the $\eta_i$ in $C\otimes_R
\Sym(\dual{A})$ and $(C,(\delta_i)_i)$ is universal with this property.
\end{enumerate}
\end{Pro}
\begin{proof}
The first two statements are equivalent since $P_a(X)$ is equal to
$s_{n, S[X]}(X - a)$, the constant term of $P_a(X)$ is $(-1)^ns_{n,S}(a)$ and
the constant term of the product in \emph{2} is
\[
\prod_i (-1)^{n_i}\delta_{i,S}(a) = (-1)^n\prod\delta_{i,S}(a)
\]

The third is equivalent with the first by Lemma
\ref{lemma-polynomial-law-generic-element} and Remark \ref{rem-ring-poly-laws}.
\end{proof}

\begin{Def}\label{def-pa}
Given $n_1,\ldots, n_t$, with $\sum n_i = n$, an algebra satisfying the
equivalent conditions in Proposition \ref{pro-def-pa} will be denoted
$A^{(n_1,\ldots, n_t)}$.
\end{Def}

To connect the $m$-closures with the constructions above we will show in
Proposition \ref{pro-poly-law-rel-with-m-closure} that the $m$-closure $A^{(m)}$
of an $R$-algebra of rank $n$ is isomorphic to $A^{(1,\ldots, 1,n-m)}$ (with $m$
ones). The following two lemmas come from Ferrand \cite[R\`egle 
$4.2.3$]{FonctNorme}.

\begin{Lem}\label{regle-f-is-g}
Let $A$ be an $R$-algebra of rank $n$ and $B$ be any $R$-algebra. Let $f,g, h
\colon \underline{A}\to \underline{B}$ be polynomial laws. Suppose we have
\[
s_n = fg = fh
\]
as polynomial laws. Then $g$ is equal to $h$.
\end{Lem}
\begin{proof}
Let $S$ be an $R$-algebra and let $x$ be in $A\otimes_R S$. Then we have
\[
s_{n,S[X]}(X + x) = f_{S[X]}(X + x) g_{S[X]}(X + x) = f_{S[X]}(X + x) h_{S[X]}(X
+ x)
\]
The left hand side is monic because $s_n$ is multiplicative. Hence its factor
$f_{S[X]}(X + x)$ is a regular element of $B\otimes_R R'[X]$ (that is: it
is not a right nor a left zero divisor), so the equality above implies
$g_{S[X]}(X + x) = h_{S[X]}(X + x)$. Then the specialization $X \mapsto 0$
implies that $g_S(x) = h_S(x)$. Since $x$ was arbitrary, the proof is complete.
\end{proof}

\begin{Lem}\label{regle-g-is-multiplicative}
Let $A$ be an $R$-algebra of rank $n$ and $B$ be any $R$-algebra. Let $f,g
\colon \underline{A}\to \underline{B}$ be polynomial laws, with $f$
multiplicative. Suppose we have
\[
s_n = fg
\]
as polynomial laws. Then $g$ is multiplicative.
\end{Lem}
\begin{proof}
As in the proof of Lemma \ref{regle-f-is-g}, for all $R\to S$ and $x, y$ in
$A\otimes_R S$ we have $f_{S[X]}(X + x)$ and $f_{S[X]}(X + y)$ are regular and
their product is
\[
f_{S[X]}(X^2 + (x+y)X + xy).
\]
So we have
\[
g_{S[X]}(X^2 + (x+y)X + xy) = g_{S[X]}(X+ x) g_{S[X]}(X+ y)
\]
because the product of both with $f_{S[X]}(X^2 + (x+y)X + xy)$ is equal to
$s_{n, S[X]}(X^2 + (x + y)X + xy)$. With the specialization $X\mapsto 0$ we get
$g(xy) = g(x)g(y)$, as we wanted to show.
\end{proof}

We are now ready to compare the algebras $A^{(1,\ldots, 1, n-m)}$ and $A^{(m)}$.

\begin{Pro}\label{pro-poly-law-rel-with-m-closure}
Let $A$ be an $R$-algebra of rank $n$. Then for all $m= 1,\ldots, n$ we have
$A^{(m)} \cong A^{(1,\ldots, 1,n-m)}$.
\end{Pro}
\begin{proof}
Recall that $A^{(m)}$ is given with $R$-algebra maps $\alpha_i\colon A\to
A^{(m)}$ for $i = 1,\ldots, m$. We define a polynomial law $\alpha\colon A\to
A^{(m)}$, multiplicative and homogeneous of degree $n-m$ and such that $\alpha
\prod_i \alpha_i$ is equal to $s_n$. Let $\gamma \in A\otimes \Sym(\dual{A})$
be the generic element of $A$. By the universal property of $A^{(m)}$ we have
that in $A^{(m)}\otimes \Sym(\dual{A})[X]$ the characteristic polynomial
$P_\gamma(X)$ is equal to
\[
Q_\gamma(X) \prod_{i=1}^{m} (X - \alpha_{i,\Sym(\dual{A})}(\gamma))
\]
for some polynomial $Q_\gamma(X)$. The constant term of $Q_\gamma(X)$ is an
element $\eta$ of $A^{(m)}\otimes \Sym(\dual{A})$, and by Lemma
\ref{lemma-polynomial-law-generic-element} this defines a unique polynomial
law $\alpha\colon \underline{A}\to \underline{A^{(m)}}$. This is
homogeneous of degree $n - m$ because $\eta$ is homogeneous of degree $n - m$
since the constant term of $P_\gamma$ is homogeneous of degree $n$ and the
constant term of $\prod (X - \alpha_{i,\Sym(\dual{A})}(\gamma) )$ is homogeneous
of degree $m$. Moreover $\alpha$ is multiplicative by Lemma
\ref{regle-g-is-multiplicative}.

We are left to show that $A^{(m)}$ with the maps $\alpha_i$ and the
polynomial law $\alpha$ has the universal property of $A^{(1,\ldots, 1,
n-m)}$. For every $R\to S$ and every element $a\in A\otimes_R S$ the
characteristic polynomial of $a$ splits as
\[
\alpha_{S[X]} (X- a)\prod_{i=1}^{m} (X - \alpha_{i,S}(a))
= \alpha_{S[X]} (X - a)\prod_{i=1}^m \alpha_{i,S[X]}(X - a).
\]
Let $B$ be an $R$-algebra given with linear maps $\beta_i\colon A\to B$ and a
polynomial law $\beta\colon A\to B$ homogeneous of degree $n-m$. Suppose that
for all $R\to S$ and all $a\in A\otimes_R A$ the polynomial $P_a(X)$ splits as
\[
\beta_{S[X]}(X - a) \prod_i \beta_{i,S[X]}(X - a)
\]
in $B[X]$. Since for all $i$ we have $\beta_{i,S[X]}(X - a) = (X -
\beta_{i,S}(a))$ by linearity, the universal property of $A^{(m)}$ gives a
unique map $\varphi\colon A^{(m)}\to B$ such that for all $i$ we have
$\varphi\circ \alpha_i = \beta_i$. Since $s_n$ is equal to both $(\varphi\circ
\alpha) \prod_i\beta_i$ and $\beta \prod_i\beta_i$, by Lemma \ref{regle-f-is-g}
we have $\varphi\circ \alpha = \beta$, so the proof is complete.
\end{proof}

We have not shown existence of the algebras $A^{(n_1,\ldots, n_t)}$. This is
done by Daniel Ferrand in \cite[\S $4.1$]{FonctNorme}, and we will not talk
about it here.


\section{Monogenic algebras}\label{sec-monogenic}

Let $R$ be a ring and let $f$ be a monic polynomial with coefficients in $R$. In
this section we will describe explicitly the closures of the $R$-algebra
$R[x]/(f)$ (a \emph{monogenic} $R$-algebra). To do this an important tool is the
following lemma.

Recall that for $P$ and $Q$ in $R[X]$ we write $P \mid Q$ for $P$ divides $Q$.

\begin{InLem}[Corollary \ref{cor-charpoly-of-a-poly}]
Let $R$ be a ring, and let $g$ be a polynomial with coefficients in $R$. Let $M$
be a free $R$-module of rank $n$ and $\alpha\in \mathrm{End}(M)$. Let $a_i\in R$
for $i=1, \ldots m$. Suppose that
\[
\prod_{i=1}^{m}(X-a_i)\mid P_\alpha(X)
\]
in $R[X]$, then also
\[
\prod_{i=1}^{m}(X-g(a_i))\mid  P_{g(\alpha)}(X)
\]
in $R[X]$.
\end{InLem}

We will need some preliminary results.

\begin{Def}
Let $R$ be a ring. We denote by $R[Z]_{n}^{\mathrm{mon}}$ the set of monic
polynomials of degree $n$ in $R[Z]$.
\end{Def}

\begin{Theo}\label{pro-def-phi-f}
Let $g \in R[X]$. For all $n \geq 0$ there exists a unique collection of maps
$(\varphi_{n, A})_{A}$ indexed over all $R$-algebras $A$, with $\varphi_{n, A}
\colon A[Z]^{\mathrm{mon}}_{n} \to A[Z]^{\mathrm{mon}}_{n}$, satisfying the
following conditions:
\begin{enumerate}
\item For all $A\to B$ the diagram
\[
\xymatrix{
A[Z]^{\mathrm{mon}}_{n} \ar[r] \ar[d]_{\varphi_{n, A}}&
B[Z]^{\mathrm{mon}}_{n} \ar[d]^{\varphi_{n, B}}\\
A[Z]^{\mathrm{mon}}_{n} \ar[r]&
B[Z]^{\mathrm{mon}}_{n}
}
\]
commutes.
\item For all $a_i\in A$ with $i = 1,\ldots, n$ we have
\[
\varphi_{n, A} \Bigg(\prod_{i=1}^n (Z - a_i)\Bigg) = \prod_{i=1}^n \left(Z -
g(a_i)\right).
\]
\end{enumerate}
Moreover, for all $n, m \geq 0$, for all $R\to A$, for all $f\in
A[Z]^{\mathrm{mon}}_{n}$ and for all $g\in A[Z]^{\mathrm{mon}}_{m}$ we have
\[
\varphi_n(f) \varphi_m(g) = \varphi_{n + m} (fg)
\]
in $A[Z]^{\mathrm{mon}}_{n +m}$.
\end{Theo}
\begin{proof}
We prove uniqueness first. Suppose a map $\varphi_n$ is given, having the
properties above. Let $T$ be the $R$-algebra $R[X_1,\ldots, X_n]$, and let
$\Delta$ be the polynomial $\prod_{i} (Z - X_i)$ in $T[Z]^{\mathrm{mon}}_{n}$.
By property \emph{2} we have
\[
\varphi_{n, T} (\Delta) = \prod_{i=1}^n (Z - g(X_i)).
\]
Note that $\Delta$ can be written as
\[
Z^n - s_1 Z^{n-1} + s_2 Z^{n-2} + \cdots + (-1)^n s_n
\]
where $s_i$ is the $i$-th elementary symmetric function in the $X_i$. Since
$\prod_{i=1}^n (Z - g(X_i))$ is invariant under permutations of the $X_i$, by
the fundamental theorem of symmetric functions (see \cite[\S 6, Th\'eor\`eme
$1$]{BourbakiAlg4}), there exist $q_i$ in $S = R[s_1,\ldots, s_n]$ for $i =
1,\ldots, n$ such that 
\[
\varphi_{n, T}(\Delta) = Z^n + \sum_{i=1}^{n} q_i Z^{n-i}.
\]

Now let $A$ be any $R$-algebra and let
\[
P = Z^n + \sum_{i=1}^n (-1)^i a_i Z^{n-1}
\]
be a polynomial in $A[Z]^{\mathrm{mon}}_{n}$. Let $\pi$ be the map $S \to
A$ sending $s_i$ to $a_i$. Note that $\pi(\Delta)$ is equal to $P$, so by
property \emph{1} we have that
\[
\varphi_{n, A}(P) = \pi(\varphi_{n, T}(\Delta))
\]
in $A[Z]^{\mathrm{mon}}_{n}$. Hence uniqueness follows.

In proving uniqueness we also have constructed a map having properties
\emph{1} and \emph{2}, so existence is proved.

We now prove the last part. Let $T$ be $R[X_1,\ldots, X_n, Y_1,\ldots, Y_m]$,
and define
\[
\Delta_X = \prod_{i=1}^n (Z- X_i), \quad \quad \Delta_Y = \prod_{i=1}^m(Z- Y_i)
\]
in $T[Z]^{\mathrm{mon}}_{n}$ and $T[Z]^{\mathrm{mon}}_{m}$ respectively.
Note that we have
\[
\Delta_X = Z^n + \sum_{i=1}^n (-1)^i s_i Z^{n-i},\quad 
\Delta_Y = Z^m + \sum_{i=1}^m (-1)^i t_i Z^{m-i}
\]
where $s_i$ is the $i$-th elementary symmetric function in the $X_i$ for $i =
1,\ldots, n$ and $t_i$ is the $i$-th elementary symmetric function in the $Y_i$
for $i = 1,\ldots, m$. Denote by $S$ the subring $R[s_1,\ldots, s_n, t_1,\ldots,
t_m]$ of $T$. Note that $\varphi_{n,T}(\Delta_X)$ and $\varphi_{n, T}(\Delta_Y)$
also have coefficients in $S$. By property \emph{2} we have
\[
\varphi_{T, n+m}(\Delta_X\Delta_Y) = \varphi_{T, n}(\Delta_X)\varphi_{T,
m}(\Delta_Y)
\]
in $T[Z]^{\mathrm{mon}}_{n+m}$, so the same holds in
$S[Z]^{\mathrm{mon}}_{n+m}$.

For an $R$-algebra $A$ and polynomials
\[
f = Z^n + \sum_{i=1}^n (-1)^i a_i Z^{n-i}, \quad g = Z^m + \sum_{i=1}^m (-1)^i
b_i Z^{m-i}
\]
in $A[Z]^{\mathrm{mon}}_{n}$ and $A[Z]^{\mathrm{mon}}_{m}$ respectively, let
$\pi$ be the map $S \to A$ sending $s_i$ to $a_i$ for all $i = 1\ldots, n$ and
$t_i$ to $b_i$ for all $i = 1, \ldots, m$. Clearly $\pi$ sends $\Delta_X$ to $f$
and $\Delta_Y$ to $g$. Hence by property \emph{1}, we have that $\varphi_{A,
n}(f) \varphi_{A, m}(g) = \varphi_{A, n+m}(fg)$. The proof is complete.
\end{proof}

\begin{Ex}
We give an example to illustrate Proposition \ref{pro-def-phi-f}. Let $R$ be
$\mathbb{Z}$ and $g$ be $X^2$. For $n = 3$ consider 
\[
(Z- a)(Z - b)(Z - c)
\]
in $A[Z]^{\mathrm{mon}}_3$ for some $R$-algebra $A$. The product
\[
(Z - a^2)(Z - b^2)(Z - c^2)
\]
is equal to 
\[
Z^3 - (a^2 + b^2 + c^2)Z^2 + (a^2b^2 + a^2c^2 + b^2c^2)Z - a^2b^2c^2Z^3
\]
and for $s_1,\,s_2,\,s_3$ the elementary symmetric functions in $a, b, c$ this
is
\[
Z^3 - (s_1^2 - 2 s_2)Z^2 + (s_2^2 - 2s_1s_3)Z - s_3^2 Z^3.
\]
So the map $\varphi_{3, A}$ sends a polynomial
\[
Z^3 - r_1 Z^2 + r_2 Z - r_3 
\]
in $A[Z]^{\mathrm{mon}}_{3}$ to 
\[
Z^3 - (r_1^2 - 2 r_2)Z^2 + (r_2^2 - 2r_1r_3)Z - r_3^2 Z^3.
\]
\end{Ex}

\begin{Lem}\label{lem-charpoly-of-a-poly}
Let $R$ be a ring and $g$ be a polynomial with coefficients in $R$. Let $M$ be
a free $R$-module of rank $n$ and $\alpha \in \mathrm{End}(M)$. Then
$\varphi_{R, n}\colon R[Z]_n^{\mathrm{mon}} \to R[Z]_n^{\mathrm{mon}}$ sends
$P_\alpha(Z)$ to $P_{g(\alpha)}(Z)$.
\end{Lem}
\begin{proof}
Suppose first that $R$ is an algebraically closed field. We can write
$P_\alpha$ as
\[
\prod_{i=1,\ldots, s}(Z - a_i)^{e_i}
\]
for some $a_i \in R$ and $e_i > 0$ and since we can write the matrix
representing $\alpha$ in Jordan normal form we have that $P_{g(\alpha)}(Z)$
is
\[
\prod_{i=1,\ldots, s}(Z - g(a_i))^{e_i}
\]
Then the statement follows from Proposition \ref{pro-def-phi-f}.

Next suppose $R$ is the ring $\mathbb{Z}[(X_{r,s})_{r,s = 1, \ldots, n}]$ and
$\alpha$ is the endomorphism given by the matrix $X=(X_{r,s})_{r,s}$.
Consider the embedding $i\colon R\to K$, of $R$ into the algebraic closure of
its quotient field. Since $\varphi_{K, n}(i(P_X))$ is $i(P_{g(X)})$ and
$i$ is injective. By property \emph{1} in Proposition \ref{pro-def-phi-f} we
have that $\varphi_{R, n}(P_X)$ is equal to $P_{g(X)}$.

Let $R$ be any ring and choose a basis of $M$. Write $\alpha$ as a matrix, say
$\alpha = (a_{r,s})_{r,s = 1, \ldots, n}$. Let $\pi$ be the map
$\mathbb{Z}[(X_{r,s})_{r,s = 1, \ldots, n}]\to R$ sending $X_{r,s}$ to
$a_{r,s}$. The map $\pi$ sends $P_X$ to $P_{\alpha}$ and $P_{g(X)}$ to
$P_{g(\alpha)}$. By property \emph{1} of $\varphi_{R, n}$ in Proposition
\ref{pro-def-phi-f} we conclude that $\varphi_R(P_\alpha)=P_{g(\alpha)}$, as we
wanted to show.
\end{proof}

\begin{Lem}\label{cor-charpoly-of-a-poly}
Let $R$ be a ring, and let $g$ be a polynomial with coefficients in $R$. Let $M$
be a free $R$-module of rank $n$ and $\alpha\in \mathrm{End}(M)$. Let $a_i\in R$
for $i=1, \ldots m$. Suppose that
\[
\prod_{i=1}^{m}(Z-a_i)\mid P_\alpha(Z)
\]
in $R[Z]$, then also
\[
\prod_{i=1}^{m}(Z-g(a_i))\mid  P_{g(\alpha)}(Z)
\]
in $R[Z]$.
\end{Lem}
\begin{proof}
By assumption there exists a polynomial $Q$ in $R[Z]$ such that 
\[
P_{\alpha}(Z) = Q(Z) \prod_{i=1}^{m}(Z-a_i).
\]
Since both $P_{\alpha}$ and $\prod_{i}(Z-a_i)$ are monic also $Q$ must be
monic. Then by Proposition \ref{pro-def-phi-f} we have that
\[
\varphi_{R,n}\left(Q(Z) \prod_{i=1}^{m}(Z-a_i)\right) =
\varphi_{R,n-m}\left(Q(Z)\right)\varphi_{R,m}\left(\prod_{i=1}^{m}
(Z-a_i)\right).
\]
By Lemma \ref{lem-charpoly-of-a-poly} we have
\[
\varphi_{R, n}\left( P_{\alpha}(Z) \right) = P_{g(\alpha)}(Z)
\]
and by property \emph{2} in Proposition \ref{pro-def-phi-f} we have
\[
\varphi_{R, m} \left(\prod_{i=1}^{m}(Z-a_i)\right) = \prod_{i=1}^{m}(Z-g(a_i)).
\]
So the claim is proved.
\end{proof}

We use Corollary \ref{cor-charpoly-of-a-poly} to give an equivalent universal
property for $A^{(m)}$ when $A$ is monogenic.

\begin{Pro}\label{theo-monogenic-univ-prop}
Let $R$ be a ring, and let $A$ be $R[x]/f(x)$, with $f$ a monic polynomial. Let
$C$ be an $R$-algebra and let $\zeta_i\colon A\to C$ for $i= 1,\ldots, m$ be
$R$-algebra maps. Suppose $\prod_i(X-\zeta_i(x))$ divides $f(X)$ in $C[X]$, and
the pair $(C, (\zeta_i)_{i=1,\ldots, m})$ is universal with this property. Then
$(C,(\zeta_i)_{i=1,\ldots, m})$ is an $m$-closure of $A$.
\end{Pro}
\begin{proof}
Note that $P_x(X)$ is equal to $f(X)$. The theorem follows from Corollary
\ref{cor-charpoly-of-a-poly}, since for every $R\to R'$ and any element $a\in
A\otimes_R R'$ there exists a polynomial $g(X)$ in $R'[X]$ such that $g(x\otimes
1)$ is $a$.
\end{proof}

We can now give an explicit form for the $m$-closure of a monogenic algebra. The
construction is the same we described in the introduction for fields; we recall
it here: let $R$ be a ring and let $f$ be a monic polynomial in $R[X]$. Put $A_0
= R$ and $f_0 = f$. Given $A_i$ and $f_i$ for $i\geq 0$, define $A_{i+1}$ to be
$A_i[x_{i+1}] / f_{i}$, and $f_{i+1}$ as the quotient
\[
f_{i+1}(X) = \frac{f_i(X)}{(X-x_{i+1})}
\]
in $A_{i+1}[X]$. We show that $A_m$ is the $m$-closure of $R[x]/f$.

\begin{Theo}\label{theo-monogenic-description}
Let $R$ be a ring, and let $f$ be a monic polynomial in $R[X]$. Let $m$ be
between $0$ and $n$ and define maps $\alpha_i\colon R[X]/(f)\to A_m$ for $i = 1,
\ldots, n$ sending $x$ to $x_i$. Then $A_m$ together with the $\alpha_i$ for $i
= 1,\ldots, m$ is the $m$-closure of $R[X]/(f)$.
\end{Theo}
\begin{proof}
Note that the $\alpha_i$ are well defined since for all $i$ we have $f(x_i) = 0$
in $A_m$. For the same reason $\prod_i(X-x_i)$ divides $f$ in $A_m[X]$, so we
only need to prove $A_m$ is universal with this property.

Let $B$ be an $R$-algebra given with maps $\beta_i\colon A\to B$ for $i = 1,
\ldots, m$ and such that $\prod_i (X - \beta_i(x))$ divides $f$ in $B[X]$. We
need to prove there exists a unique map $\varphi\colon A_m\to B$ such that for
every $i$ we have $\beta_i = \varphi\circ \alpha_i$.

We prove existence by induction on $m$: if $m = 0$, then $A_0 = R$ and the
unique map $R\to B$ is the required one. Suppose we have a map
$\varphi\colon A_{m-1}\to B$ satisfying $\beta_i = \varphi\circ \alpha_i$ for
$i=1,\ldots, m-1$. Consider the induced map $A_{m-1}[X]\to B[X]$ and compose
with the map $B[X]\to B$ sending $X$ to $\beta_m(x)$. Since $\beta_m(x)$ is a
root of $f_{m-1}$ in $B[X]$, this map factors through $A_m$, giving the required
map.

Any $R$-algebra map with this property sends $x_i$ to $\beta_i(x)$ and hence
coincides with $\varphi$ because the $x_i$ generate $A_m$, so the proof is
complete.
\end{proof}

\begin{Cor}
Let $R$ be a ring and $f$ be a monic polynomial in $R[X]$ of degree $n$. Let
$A$ be the $R$-algebra $R[X]/(f)$. Then for all $0\leq m \leq n$ the algebra
$A^{(m)}$ is free of the expected rank $n(n-1)\cdots (n-m+1)$.
\end{Cor}
\begin{proof} 
This follows by induction since the description given above tells us that
$A^{(i)}$ is free of rank $n - i + 1$ over $A^{(i-1)}$ for all $1\leq i\leq n$.
\end{proof}

We will see in the next section (see Remark \ref{Am-not-locally-free}) that for
general free algebras it is not true that $A^{(m)}$ is even locally free.


\section{Examples and explicit computations}\label{galois-closure-examples}

In this section we will see some explicit computations of $S$-closures, together
with some techniques to simplify the computations.

Recall that for $A$ an $R$-algebra and $S$ any set we denote by $A^{\otimes S}$
the tensor power of $A$ indexed over $S$, defined in Definition
\ref{def-tensor-power-S}, and for each $s\in S$ we denote by $\varepsilon_s$ the
natural map.

\begin{Lem}\label{lem-sum-a-constant}
Let $R$ be a ring and $A$ be an $R$-algebra of rank $n$. Let $r$ be in $R$
and $a\in A$. For any finite set $S$ define
\[
\Delta_a = \prod_{s\in S}(X - \varepsilon_s(a)).
\]
Then $\Delta_a(X) \mid P_a(X)$ if and only if $\Delta_{a+r}(X) \mid P_{a +
r}(X)$.
\end{Lem}
\begin{proof}
Note that $P_{a +r}(X) = P_a (X - r)$ and $\Delta_{a + r}(X) = \Delta_{a}(X -
r)$. The statement is then clear.
\end{proof}

\begin{Rem}\label{rem-div-one-by-one}
Let $P(X)$ be a polynomial with coefficients in a ring $R$ and let $r_1,\ldots,
r_m$ be in $R$ with $m$ at most equal to the degree of $P$. Define
\[
\Delta(X)= \prod_{i=1}^m (X - r_i).
\]
We find an expression for the remainder of the division of $P(X)$ by
$\Delta(X)$. Write
\[
P(X) = Q(X) \Delta(X) + T(X)
\]
with degree of $T(X)$ smaller than $m$. Note that $Q(X)$ and $T(X)$ are
uniquely determined by these properties.

Let $Q_0(X)$ be $P(X)$ and for all $i \geq 1$ define
\[
Q_{i}(X) = \frac{Q_{i-1}(X)}{ (X - r_i)}.
\]
The remainder in the division of $P(X)$ by $(X - r_1)$ is $P(r_1)$, so we can
write:
\[
P(X) = Q_1(X) (X - r_1) + Q_0(r_1)
\]
Dividing $Q_1(X)$ by $(X - r_2)$ we get:
\[
Q_1(X) = Q_2(X) (X - r_2) + Q_1(r_2)
\]
so that we have
\[
P(X) = Q_2(X) (X - r_1)(X - r_2) + Q_1(r_2)(X - r_1) + Q_0(r_1)
\]
and by uniqueness, the quotient and remainder of the division of $P(X)$ by
$(X-r_1)(X-r_2)$ are $Q_2(X)$ and $Q_1(r_2)(X - r_1) + Q_0(r_1)$ respectively.

In the same way we can determine $Q(X)$ and $T(X)$ as follows:
\[
P(X) = Q_m(X) \Delta(X) + \sum_{i=0}^{m-1}\left(Q_i(r_{i+1})\prod_{k=1}^{i}
(X - r_k)\right).
\]
This will be used to compute $A^{(S)}$.
\end{Rem}

\subsection*{\texorpdfstring{Explicit generators of $J^{(S)}$}{Explicit
generators of J\^{}{(S)}}}

Let $R$ be a ring and $A$ be an $R$-algebra of rank $n$. As we proved in
Proposition \ref{pro-gal-constr}, the $S$-closure of $A$ is the quotient of
$A^{\otimes S}$ by an ideal that we denoted $J^{(S)}$. For convenience we recall
the definition of $J^{(S)}$ when $A$ is \emph{free} of rank $n$ (see Notation
\ref{not-ideal} and Example \ref{ex-S-closure-free}). Let $e_1,\ldots, e_n$ be a
basis of $A$ and let $X_1,\ldots, X_n$ be the dual basis. Let $\gamma \in
A[X_1,\ldots, X_n]$ be the generic element of $A$ and let $P_\gamma$ be its
characteristic polynomial. Let $\varepsilon_s$ be the natural map $A \to
A^{\otimes S}$ and define
\[
\Delta_\gamma(Z) = \prod_{s\in S} (Z - \varepsilon_s \otimes \mathrm{Id}
(\gamma))
\]
in $A^{\otimes S}[X_1,\ldots, X_n, Z]$.  For a multi-index $I = (i_1,\ldots,
i_n)$ write $X^I$ for $X_1^{i_1}\cdots X_n^{i_n}$. Let $T_\gamma$ be the
remainder in the division of $P_\gamma$ by $\Delta_\gamma$ in the ring
$A^{\otimes S} [X_1,\ldots, X_n, Z]$. Write the coefficient of $Z^i$ in
$T_\gamma$ as $\sum a_{I,i}X^I$ for $I$ ranging over multi-indices as above.
Then $J^{(S)}$ is the ideal of $A^{\otimes S}$ generated by the $a_{I,i}$.

We compute a set of generators of $J^{(S)}$ in case $A$ is a connected algebra
of rank $n$ over a field $K$. Recall from Lemma
\ref{lem-finite-algebra-over-field} that in this case $A$ is local with
nilpotent maximal ideal. For simplicity we will also take $S = \{1,\ldots, m\}$,
and suppose $m \leq n$.

Choose a basis $e_1,e_2,\ldots, e_n$ of $A$ such that $e_1 = 1$, and $e_2,
\ldots, e_n$ are in $\mathfrak{m}$, and let $X_1,\ldots, X_n$ be the dual basis.
The element $\gamma \in A[X_1,\ldots, X_n]$ can be written as $\sum e_iX_i$, as
we have seen in Example \ref{ex-gen-elem} and Example \ref{ex-gen-elem-two}.
Denote by $\gamma'$ the difference $\gamma - e_1X_1$, and define in the obvious
way the polynomials
\[
P_{\gamma'}(Z), \quad \Delta_{\gamma'}(Z), \quad T_{\gamma'}(Z)
\]
in $A[X_1,\ldots, X_n, Z]$. Write $\sum {a}_{I,i}'X^I$ for the coefficient of
$Z^i$ in $T_{\gamma'}(Z)$ and let ${J^{(m)}}'$ be the ideal of $A^{\otimes S}$
generated by the $a_{I,i}'$. By Lemma \ref{lem-sum-a-constant} we have that
$\Delta_{\gamma'}$ divides $P_{\gamma'}$ if and only if $\Delta_{\gamma}$
divides $P_{\gamma}$ so we can repeat the argument in Proposition 
\ref{pro-gal-constr} with ${J^{(m)}}'$ and we get that $A^{\otimes
m}/{J^{(m)}}'$ is the $m$-closure of $A$.

Let $\gamma_i'$ be the image of $\gamma'$ via the $i$-th natural map from
$A[X_1,\ldots, X_n]$ to the tensor power $A^{\otimes m}[X_1,\ldots, X_n]$. By
Remark \ref{rem-div-one-by-one} we can write
\[
T_{\gamma'}(Z) = \sum_{i=0}^{m-1}\left(Q_i(\gamma_{i+1}')\prod_{k=1}^{i} (Z -
\gamma_k')\right).
\]
and the ideal generated by the coefficients of $T_{\gamma'}$ is equal to the
ideal generated by the coefficients of $Q_i(\gamma_{i+1}')$.

Note that since $e_2,\ldots, e_n$ are nilpotent $P_{\gamma'}(Z)$ is equal
to $Z^n$. For $i \geq 0$ and $k = 1,\ldots, n$ define $d_i(\gamma_1',\ldots,
\gamma_k')$ to be the sum of all monomials of degree $i$ in the variables
$\gamma_1', \ldots, \gamma_k'$. Then we have
\begin{equation}
Q_i(\gamma_{i+1}') = d_{n-i}(\gamma_1',\ldots, \gamma_{i+1}') \in A^{\otimes m}
[X_1,\ldots, X_n] \label{formula-generators}
\end{equation}
for all $i  = 1,\ldots, m-1$. So $A^{(m)}$ is the quotient of $A^{\otimes m}$ by
the ideal generated by the coefficients of these polynomials.

\begin{Ex}\label{ex-not-bcinv}
For an $R$-algebra $A$ of rank $n$ in Remark \ref{rem-naive-def} we defined
$G^{(m)}(A/R)$ to be the quotient of $A^{\otimes m}$ by the ideal $I$ generated
by the coefficients of the remainder in the division of $P_a$ by $\prod_{i=1}^m
(X - \varepsilon_i(a))$, with $a$ ranging over $A$. In Remark
\ref{rem-naive-def} we said that this construction, modeled on the one by
Bhargava and Satriano, does not in general commute with base change for $m<n$.
We are going to give an example here.

Let $A$ be the $\mathbb{F}_2$-algebra $\mathbb{F}_2 [X_1,\ldots,X_4]/
(X_1,\ldots, X_4)^2$. We show that $G^{(3)}(A/\mathbb{F}_2)$ has dimension
$110$, while the dimension of $G^{(3)}(A\otimes_{\mathbb{F}_2}\mathbb{F}_4
/\mathbb{F}_4)$ is at most $109$, so this construction does not commute with
base change.

Let $\mathfrak{m}$ be the maximal ideal of $A$. Since any element $x$ in $A$
can be written as $x = r + m$, with $r\in \mathbb{F}_2$ and $m$ in
$\mathfrak{m}$, by Lemma \ref{lem-sum-a-constant} we can generate the ideal $I$
with the coefficients of the remainder in the division of $P_a$ by
$\prod_{i=1}^m (X - \varepsilon_i(a))$, with $a$ ranging over $\mathfrak{m}$,
instead of $A$. Since the characteristic polynomial of $a\in \mathfrak{m}$ is
$X^5$, the computations we did for $Q_i(\gamma_{i+1}')$ apply here as well. Then
we have
\[
I = \left\langle \sum_{i+ j = 4} a^i\otimes a^j\otimes 1,\, \sum_{i+j+k =
3}a^i\otimes a^j\otimes a^k \mid a\in \mathfrak{m} \right\rangle.
\]
But since multiplication of any two elements in $\mathfrak{m}$ is zero, most of
the terms in the expressions above are zero and we get
\[
I= \left\langle a\otimes a\otimes a\mid a\in \mathfrak{m}\right\rangle.
\]
Moreover for all $x = r + m \in A^{\otimes 3}$ we have that $x (a\otimes
a\otimes a)$ is equal to $r (a\otimes a\otimes a)$, since $a$ multiplied by any 
element in $\mathfrak{m}$ is zero, so the ideal $I$ equals the 
$\mathbb{F}_2$-vector space generated by the $a\otimes a\otimes a$. Then $I$ is
a vector space generated by $15$ non-zero vectors. So its dimension is at most
$15$, and the dimension of $G^{(3)}(A/\mathbb{F}_2)$ is at least $125 - 15 =
110$ (in fact it is $110$).

Now consider the same construction over $\mathbb{F}_4$. Here $G^{(3)}
(A\otimes_{\mathbb{F}_2}\mathbb{F}_4/\mathbb{F}_4)$ is equal to the quotient of
$A^{\otimes 3}\otimes_{\mathbb{F}_2}\mathbb{F}_4$ by the ideal
\[
J = \left\langle a\otimes a\otimes a\mid a\in \mathfrak{m}\otimes_{\mathbb{F}_2}
\mathbb{F}_4\right\rangle
\]
but in this case the dimension of $J$ is at least $16$. In fact, consider the
projection $\pi$ from $(\mathfrak{m}\otimes_{\mathbb{F}_2}
\mathbb{F}_4)^{\otimes 3}$ to the quotient $\mathrm{Sym}^3 (\mathfrak{m}
\otimes_{\mathbb{F}_2}\mathbb{F}_4)$, which is isomorphic to $(\mathbb{F}_4[T_1,
T_2, T_3, T_4])_3$, the $20$-dimensional vector space of homogeneous polynomials
of degree $3$ in $4$ variables over $\mathbb{F}_4$. The subspace $\pi(W)$ is
generated by third powers of polynomials of degree $1$. Consider
\[
\left(\sum_{i= 1}^4 {r_iT_i}\right)^3 =
\left(\sum_{i= 1}^4 {r_i^2 T_i^2}\right) \left(\sum_{i= 1}^4 {r_iT_i}\right)=
\sum_{i,j}r_i^2r_jT_i^2T_j
\]
with $r_i$ in $\mathbb{F}_4$. A computation shows that a basis for $\pi(W)$ is
\[
\begin{split}
&\{T_1^3,\, T_2^3,\, T_3^3,\,T_4^3, \\
&T_1^2T_2,\,T_1^2T_3,\, T_1^2T_4,\, T_2^2T_3,\, T_2^2T_4,\, T_3^2T_4,\\
&T_1T_2^2,\, T_1T_3^2,\, T_1T_4^2,\, T_2T_3^2,\, T_2T_4^2,\, T_3T_4^2\}
\end{split}
\]
so it has dimension $16$. Hence there exists a surjective map from $W$ to a
space of dimension $16$, and the dimension of $W$ is at least $16$. Then the
algebra $G^{(3)}(A\otimes_{\mathbb{F}_2}\mathbb{F}_4/\mathbb{F}_4)$ has
dimension at most $125 - 16 = 109$ (in fact the dimension is $105$). So this
construction does not commute with base change.
\end{Ex}

\subsection*{\texorpdfstring{Algorithm for computing $S$-closures}{Algorithm for
computing S-closures}}

As we pointed out in Example \ref{ex-S-closure-free} the construction using the
generic element given in Section \ref{sec-S-closures} is completely explicit. To
illustrate this, we give here an implementation of the construction using MAGMA
\cite{MAGMA}. The following code, has as input a finite commutative algebra $A$
over a field $K$, given with generators and relations, and an integer $m$. As
output it returns the $m$-closure of $A$. In the implementation, we use the
MAGMA command $\texttt{Algebra}$, which computes a basis and the multiplication
table of the input algebra $A$.

{\footnotesize
\begin{verbatim}
// Function: mClosure. Given an affine algebra A over a field K and an 
// integer m returns the m-closure of A.

mClosure := function(A, m)

// Preliminaries
relations := Generators( DivisorIdeal( A ) ); // Relations in A
gen := Rank( A ); // Number of generators in A
K := CoefficientRing( A ); // Base ring of A
// The ring A as a K-vector space (B) and the bijection A -> B
B, f := Algebra( A );
fInv := Inverse( f );
d := Dimension( B ); // The dimension of B over K

// Vector with multiplication matrices for basis elements in B
Mat := [ Matrix( K, d, d, [ ElementToSequence( B.i * B.j ) : i in
[1..d] ] ) : j in [1..d] ];

// Construction of the m-th tensor power of A
// Polynomial ring for defining the tensor power
PolyRing := PolynomialRing( K, m * gen );
// Sequence with variables in PolyRing
x := [ [ PolyRing.(i+m*j-m) : i in [1..m] ] : j in [1..gen] ];
// Relations between elements in x
TensorPowerAlgebra := quo< PolyRing | [ Evaluate( rel, [ x[i][j] :
i in [1..gen] ]) : j in [1..m ], rel in relations ] >;

// Embeddings of A into the tensor power
alpha := [ hom< A -> TensorPowerAlgebra | [ TensorPowerAlgebra!(x[i][j])
: i in [1..gen] ] > : j in [1..m] ];

// The ring where the coefficients of Delta and the characteristic
// polynomial of the generic element live
GenericRing := PolynomialRing( TensorPowerAlgebra, d );
// The matrix of multiplication by the generic element
MultMat := Matrix( GenericRing, d, d, [ [ &+ [ GenericRing.n *
GenericRing!Mat[n][j][i] : n in [1..d] ] : i in [1..d] ] : 
j in [1..d] ] );

// Construction of Delta and characteristic polynomial of gamma
// Polynomial ring with coefficients in GenericRing
Polynom<X> := PolynomialRing( GenericRing );
// a sequence gamma with gamma[i] the i-th embedding of gamma in GenericRing
gamma := [ &+ [ GenericRing!(alpha[i](fInv(Basis(B)[j]))) *  
GenericRing.(j) : j in [1..d] ] : i in [1..m] ];
// The polynomial Delta
if m ne 0 then
Delta := &* [ X - gamma[i] : i in [1..m] ];
else Delta := Polynom!1;
end if;
// Characteristic polynomial of the generic element
CharPoly := CharacteristicPolynomial( MultMat );

// Computation of the coefficients of the remainder in the division of
// CharPoly by Delta (the relations to quotient out)
CoeffRemainder := Coefficients(Polynom!CharPoly mod Delta);
ClosureRelations := [];
for i in [1..#CoeffRemainder] do
ClosureRelations := ClosureRelations cat Coefficients(CoeffRemainder[i]);
end for;

// Final computation of A^{(m)}
Am := quo< TensorPowerAlgebra | ClosureRelations >;
return Am;
end function;
\end{verbatim}}

\subsection*{\texorpdfstring{Dimensions of $A^{(m)}$ for $A$ of
dimension $\leq 6$ over an algebraically closed field}{Dimensions of
A\^{}{(m)} for A of dimension <=6 over an algebraically closed field}}

We used the code provided above to compute the dimension of $m$-closures of
all algebras of dimension $n\leq 6$ over an algebraically closed field $K$ of
characteristic $0$. We use the classification in \cite{PoonenIsoType}, by Bjorn
Poonen. Note that for $n \geq 7$ there exist infinitely many non-isomorphic
$K$-algebras (see references in \cite{PoonenIsoType}).

Note that we can reduce the computation to local algebras, since by Lemma
\ref{lem-finite-algebra-over-field} every finite $K$-algebra is a product of
local ones and by the product formula (see Theorem \ref{theo-prod-formula}) we
have:
\[
\dim (A_1\times A_2)^{(m)} = \sum_{k=0}^{m} \binom{m}{k}\dim A_1^{(k)}
\dim A_2^{(m-k)}.
\]

\begin{Exp} \label{exp-table}
Let $A$ be a local $K$-algebra with maximal ideal $\mathfrak{m}$. In Table
\ref{table-up-to-5} we consider all algebras with $n\leq 5$. We list
the dimension of $A$, the sequence $d = (\dim( \mathfrak{m}^i/
\mathfrak{m}^{i+1}))_{i\geq 1}$, the ideal defining $A$, and the dimension of
$A^{(m)}$ for $m = 2, \ldots, n-1$. We did not write the dimension of the
$n$-closure and of the $1$-closure: by Theorem \ref{theo-S-n-1} the
$(n-1)$-closure is isomorphic to the $n$-closure, and by number \emph{3} of
Proposition  \ref{pro-Sclo-special-cases} the $1$-closure is isomorphic to $A$.
The algorithm was run in these cases as well, and confirms the results (or
rather the results verify the algorithm). The table for dimension $6$ is at the
end (Table \ref{table-dimension-6}, page \pageref{table-dimension-6}).
\end{Exp}

{
\begin{table}[h!t]
\small\centering
\begin{tabular}{c|c|c||c|c|c}
\multicolumn{3}{c||}{}&\multicolumn{3}{c}{$\dim A^{(m)}$}\\
$n$ &
$d$&
Ideal&
$m=2$&
$m=3$&
$m=4$\\
\hline\hline
\rule{0pt}{3ex}$3$& $(1,\, 1)$ & $(x^3)$ &  $6$\\
& $(2)$& $(x, y)^2$ & $6$\\
\hline
\rule{0pt}{3ex}$4$& $(1,\,1,\,1)$ & $(x^4)$ & $12$& $24$\\
& $(2,1)$ & $(x^3, y^2, xy)$ & $13$& $26$\\
& & $(x^2, y^2)$ & $13$& $26$\\
& $(3)$ & $(x, y, z)^2$ & $16$& $32$\\
\hline
\rule{0pt}{3ex}$5$ & $(1,\, 1,\, 1,\, 1)$ & $(x^5)$ & $20$& $60$& $120$\\
& $(2,\, 1,\, 1)$ & $(x^2, y^4, xy)$ & $21$& $65$& $130$\\
& & $(x^2 + y^3, xy)$ & $21$& $65$& $130$\\
& $(2,\, 2)$ & $(x^3, y^2, x^2y)$ & $22$& $70$& $140$\\
& & $(x^3, y^3, xy)$ & $22$& $70$& $140$\\
& $(3, 1)$ & $(x^2, y^2, z^2, xy, xz)$ & $24$& $80$& $160$\\
& & $(x^2, y^2, z^3, xy, xz, yz)$ & $24$& $84$& $170$\\
& & $(x^2, y^2, xz, yz, xy + z^2)$ & $24$& $80$& $160$\\
& $(4)$ & $(x, y, z, w)^2$ & $25$& $105$& $220$\\
\end{tabular}
\caption{Dimension up to $5$, see Explanation \ref{exp-table}
\label{table-up-to-5}}
\end{table}
}

\begin{Rem}
Let $A$ be a $K$-algebra of dimension $n$. In Remark \ref{rem-expected-rank} we
called $n(n-1) \cdots (n-m+1)$ the \emph{expected rank} of the $m$-closure of
$A$. From the tables it is apparent that it is in fact quite rare that the rank
of the $m$-closure of a $K$-algebra $A$ is equal to the expected one: for the
algebras considered here this happens only for the $3$-dimensional $K$-algebra
$K[X,Y]/(X,Y)^2$, and for the cases already treated in the preceding sections,
like for monogenic algebras (see Theorem \ref{theo-monogenic-description}), and
for $m = 0,1$ or $m > n$ (see Proposition \ref{pro-Sclo-special-cases}). It
would be interesting to know whether the dimension of $A^{(m)}$ can be smaller
than the expected one, but no examples were found, nor a proof that this cannot
happen. This was asked for the Galois closure in \cite[Question $3$]{BhargSat}.
\end{Rem}

\begin{Rem}
The argument for the proof of Theorem $8$ in \cite{BhargSat} applies to
$m$-closures as well. So for all $m$ the dimension of the $m$-closure of a
$K$-algebra of dimension $n$ is at most the dimension of the $m$-closure of
$K[X_1,\ldots, X_{n-1}]/(X_1,\ldots, X_{n-1})^2$.
\end{Rem}

\begin{Rem}\label{Am-not-locally-free}
Note that Table \ref{table-up-to-5} and Table \ref{table-dimension-6} are
specific to characteristic $0$. In characteristic $2$ the algebra of dimension
$4$ defined by the ideal $(x^2, y^2)$ has $2$-closure of dimension $16$ and
$3$-closure of dimension $32$, and not $13$ and $26$ as in the table. In
particular (see also the discussion after theorem $6$ in \cite{BhargSat}), this
implies there exist free $\mathbb{Z}$-algebras with a non locally free
$m$-closure.
\end{Rem}

\subsection*{More computations about dimensions}

We list here some more results about the dimension of $m$-closures. The
$m$-closure of an algebra can sometimes be equal to the whole $A^{\otimes
m}$. Here is a sufficient condition.

\begin{Pro}
Let $A$ be a local algebra over a field $K$, with residue field $K$. Suppose
there exists $t\geq 2$ such that for all $a$ in the maximal ideal of $A$ we have
$a^t = 0$. Let $n$ be the dimension of $A$ and suppose $n - tm \geq 0$. Then
$A^{(m)}$ is $A^{\otimes m}$.
\end{Pro}
\begin{proof}
Let $K'$ be any $K$-algebra and let $a$ be in the maximal ideal of $A\otimes_K
K'$. Let $\varepsilon_i$ be the natural map $A\otimes_K K'\to A^{\otimes
m}\otimes_K K'$. Since $a^t$ is zero we have
\[
X^t = (X - \varepsilon_i(a)) \sum_{k=0}^{t-1} \varepsilon_i(a)^k X^{t-k-1}
\]
in $A^{\otimes m}\otimes_K K'[X]$ for $i = 1,\ldots, m$. The characteristic
polynomial of $a$ is $X^n$, since $a$ is nilpotent. We can write
\[
X^n = X^t \cdots X^t X^{n-tm} = \prod_{i=1}^{m}(X - \varepsilon_i(a)) X^{n-tm}
\prod_{i=1}^{m}\left(\sum_{k=0}^{t-1} \varepsilon_i(a)^k X^{t-k-1}\right).
\]
In particular $\prod_i (X - \varepsilon_i(a))$ divides the characteristic
polynomial of $a$ in $A^{\otimes m}\otimes_K K'[X]$. This proves our claim.
\end{proof}

The algebras of dimension $n$ (over a field) defined in the following
proposition have $2$-closure of dimension $n^2- 3$, which is bigger than the
expected $n^2 - n$.

\begin{Pro}
Let $K$ be a field. For every $n\geq 3$ let $A$ be the following $n$-dimensional
$K$-algebra:
\[
\begin{split}
A &= K[x,y]/(x^{\frac{n}{2}+1}, y^{\frac{n}{2}}, xy), \textrm{ for } n
\textrm{ even} \\
A &= K[x,y]/(x^{\frac{n+1}{2}}, y^{\frac{n+1}{2}}, xy), \textrm{ for } n
\textrm{ odd.}
\end{split}
\]
Then $A^{(2)}$ has dimension $n^2 - 3$.
\end{Pro}
\begin{proof}
Let $\gamma \in A[Z_1,\ldots, Z_n]$ be the generic element of $A$ and let
$\gamma'$ be $\gamma - Z_1$, the generic element of the maximal ideal of $A$.
Denote by $\gamma_i'$ for $i = 1, 2$ the image of $\gamma'$ via the base change
of the natural maps $\varepsilon_i\colon A\to A\otimes_R A$. Let $I$ be the
ideal in $A\otimes_R A$ such that $A\otimes_R A/I$ is $A^{(2)}$. By Formula
\ref{formula-generators}, page \pageref{formula-generators}, the ideal $I$ is
generated by
\[
{\gamma_1'}^{n-1} + {\gamma_1'}^{n-2} \gamma_2' + \cdots + \gamma_1'
{\gamma_2'}^{n-2} + {\gamma_2'}^{n-1},
\]
and by the relations in $A$ this is equal to
\[
\begin{split}
 {\gamma_1'}^{\frac{n}{2}}{\gamma_2'}^{\frac{n}{2} -1} &+
{\gamma_1'}^{\frac{n}{2} -1}{\gamma_2'}^{\frac{n}{2}},\, \textrm{ for } n
\textrm{ even} \\
 {\gamma_1'}^{\frac{n-1}{2}}&{\gamma_2'}^{\frac{n-1}{2}},\, \textrm{ for } n
\textrm{ odd.} \\
\end{split}
\]
Suppose now $n$ is odd and let $t$ be $\frac{n-1}{2}$. Write 
\[
\gamma'= \sum_{i=1}^{t}x^iZ_i + \sum_{i=1}^t y^i Z_{t +i} 
\]
Then $\gamma'^t$ is $x^tZ_1^t +y^tZ_t^t$, and the ideal $I$ is generated by
\[
x^{t} \otimes x^t, \quad y^t \otimes y^t, \quad x^t \otimes y^t + y^t \otimes
x^t.
\]
It is now easy to see that the dimension of $I$ is $3$. For $n$ even let $t =
\frac{n}{2}$ and write
\[
\gamma'= \sum_{i=1}^{t}x^iZ_i + \sum_{i=1}^{t-1} y^i Z_{t-1 +i}.
\]
Computing ${\gamma'}_1^{t}{\gamma_2'}^{t-1} + {\gamma'}_1^{t-1}{\gamma_2'}^{t}$
we see that $I$ is generated by
\[
x^t\otimes y + y\otimes x^t,\quad x^t \otimes x + x\otimes x^t.
\]
Again one can compute the dimension of $I$, and this is again $3$.
\end{proof}

It would be interesting to know the possible values for the dimensions of
$A^{(m)}$. We conclude with the promised table with the dimensions of the
$m$-closures of algebras of dimension $6$ over an algebraically closed field of
characteristic $0$.

{\enlargethispage{2cm}
\begin{table}[ht]
\small\centering
\begin{tabular}[t]{@{}c@{}|@{}c@{} || @{}c@{} | @{}c@{} | @{}c@{} | @{}c@{}
}
\multicolumn{2}{c||}{}&
\multicolumn{4}{c}{$\dim A^{(m)}$}\\
$d$&
Ideal&
$\phantom{\_}m{=}2\phantom{\_}$&
$\phantom{\_}m{=}3\phantom{\_}$&
$\phantom{\_}m{=}4\phantom{\_}$&
$\phantom{\_}m{=}5\phantom{\_}$\\
\hline\hline
\rule{0pt}{3ex}$(1,\,1,\,1,\,1,\,1)$& $(x^6)$ & $30$& $120$& $360$& $720$\\ 
$(2,\,1,\,1,\,1)$& $(x^2, y^5, xy)$ & $31$  &  $129$  &  $393$ & $785$\\
& $(x^2 + y^4, xy)$ &  $31$  &  $129$  &  $393$  & $785$\\
$(2,\, 2,\, 1)$& $(xy, x^3, y^4)$&  $33$  &  $141$  &  $436$  & $870$\\
& $(xy, x^3 + y^3)$ &  $33$  &  $138$  &  $422$  & $840$\\
& $(x^2, xy^2, y^4)$&  $33$  &  $145$  &  $453$  & $905$\\
& $(x^2, y^3)$ &  $33$  &  $142$  &  $439$  &  $875$\\
& $(x^2 + y^3, xy^2, y^4)$ &  $33$  &  $144$  &  $450$  &  $900$\\
$(2,\,3)$& $(x, y)^3$ &  $36$  &  $165$  &  $539$  &  $1085$\\
$(3,\,1,\,1)$& $(x^2, xy, y^2, xz, yz, z^4)$ & $34$ & $160$ & $520$ & $1045$\\
& $(x^2, xy, y^2 + z^3, xz, yz, z^4)$ &  $34$  &  $154$  &  $488$ & $975$\\
& $(x^2, xy + z^3, y^2, xz, yz, z^4)$ &  $34$  &  $154$  &  $488$ & $975$\\
$(3,\,2)$& $(xy, yz, z^2, y^2 - xz, x^3)$ & $36$ & $168$ & $540$ & $1080$\\
& $(xy, z^2, xz - yz, x^2 + y^2 - xz)$&  $36$  &  $168$  &  $540$ & $1080$\\
& $(x^2, xy, xz, y^2, yz^2, z^3)$ &  $36$  &  $176$  &  $587$  & $1185$\\
& $(x^2, xy, xz, yz, y^3, z^3)$ &  $36$  &  $172$  &  $570$ & $1150$\\
& $(xy, xz, y^2, z^2, x^3)$ &  $36$  &  $168$  &  $538$  &  $1075$\\
& $(xy, xz, yz, x^2 + y^2 - z^2)$ &  $36$  &  $164$  &  $516$  & $1028$ \\
& $(x^2, xy, yz, xz + y^2 - z^2)$&  $36$  &  $164$  &  $518$  & $1033$\\
& $(x^2, xy, y^2, z^2)$ &  $36$  &  $170$  &  $546$  &  $1090$\\
$(4, 1)$& $(x^2, y^2, z^2, xy, xz, xw, yz, yw, zw, w^3)$&  $36$  & $195$  & 
$707$  &  $1457$\\
& $(x^2, y^2, z^2, w^2, xy, xz, xw, yz, yw)$&  $36$  &  $193$  & $667$ &
$1352$\\
& $(x^2, y^2 + zw, z^2, w^2, xy, xz, xw, yz, yw)$ &  $36$  &  $193$ & $661$  & 
$1334$\\
& $(x^2, y^2, z^2, w^2, xy - zw, xz, xw, yz, yw)$ &  $36$  &  $193$ & $661$  & 
$1334$\\
$(5)$& $( x, y, z, w, v )^2$ &  $36$  &  $216$  &  $876$  &  $1875$\\
\end{tabular}
\caption{Dimension $6$, see Explanation \ref{exp-table}
\label{table-dimension-6}}
\end{table}
}




\end{document}